\documentclass[a4paper, british]{amsart}

\date{Janvier 2019}

\ifdefined\subtitle

\date{Janvier 2019}
\bbkannee{71\textsuperscript{e} année, 2018--2019} 
\bbknumero{1158} 
\else
\usepackage{libertine}
\usepackage[libertine]{newtxmath}
\date{\today}
\fi

\usepackage{amssymb}
\usepackage{babel}
\usepackage{enumitem}
\usepackage{hyperref}
\usepackage[utf8]{inputenc}
\usepackage{newunicodechar}
\usepackage{mathtools}
\usepackage{varioref}
\usepackage[arrow,curve,matrix]{xy}

\usepackage{colortbl}
\usepackage{graphicx}
\usepackage{tikz}

\usepackage{mathrsfs}

\definecolor{linkred}{rgb}{0.7,0.2,0.2}
\definecolor{linkblue}{rgb}{0,0.2,0.6}

\numberwithin{figure}{section}

\usepackage[hyperpageref]{backref}

\setdescription{labelindent=\parindent, leftmargin=2\parindent}
\setitemize[1]{labelindent=\parindent, leftmargin=2\parindent}
\setenumerate[1]{labelindent=0cm, leftmargin=*, widest=iiii}

\newunicodechar{α}{\ensuremath{\alpha}}
\newunicodechar{β}{\ensuremath{\beta}}
\newunicodechar{χ}{\ensuremath{\chi}}
\newunicodechar{δ}{\ensuremath{\delta}}
\newunicodechar{ε}{\ensuremath{\varepsilon}}
\newunicodechar{Δ}{\ensuremath{\Delta}}
\newunicodechar{η}{\ensuremath{\eta}}
\newunicodechar{γ}{\ensuremath{\gamma}}
\newunicodechar{Γ}{\ensuremath{\Gamma}}
\newunicodechar{ι}{\ensuremath{\iota}}
\newunicodechar{κ}{\ensuremath{\kappa}}
\newunicodechar{λ}{\ensuremath{\lambda}}
\newunicodechar{Λ}{\ensuremath{\Lambda}}
\newunicodechar{ν}{\ensuremath{\nu}}
\newunicodechar{μ}{\ensuremath{\mu}}
\newunicodechar{ω}{\ensuremath{\omega}}
\newunicodechar{Ω}{\ensuremath{\Omega}}
\newunicodechar{π}{\ensuremath{\pi}}
\newunicodechar{Π}{\ensuremath{\Pi}}
\newunicodechar{φ}{\ensuremath{\phi}}
\newunicodechar{Φ}{\ensuremath{\Phi}}
\newunicodechar{ψ}{\ensuremath{\psi}}
\newunicodechar{Ψ}{\ensuremath{\Psi}}
\newunicodechar{ρ}{\ensuremath{\rho}}
\newunicodechar{σ}{\ensuremath{\sigma}}
\newunicodechar{Σ}{\ensuremath{\Sigma}}
\newunicodechar{τ}{\ensuremath{\tau}}
\newunicodechar{θ}{\ensuremath{\theta}}
\newunicodechar{Θ}{\ensuremath{\Theta}}
\newunicodechar{ξ}{\ensuremath{\xi}}
\newunicodechar{Ξ}{\ensuremath{\Xi}}
\newunicodechar{ζ}{\ensuremath{\zeta}}

\newunicodechar{ℓ}{\ensuremath{\ell}}
\newunicodechar{ï}{\"{\i}}

\newunicodechar{𝔸}{\ensuremath{\bA}}
\newunicodechar{𝔹}{\ensuremath{\bB}}
\newunicodechar{ℂ}{\ensuremath{\bC}}
\newunicodechar{𝔻}{\ensuremath{\bD}}
\newunicodechar{𝔼}{\ensuremath{\bE}}
\newunicodechar{𝔽}{\ensuremath{\bF}}
\newunicodechar{ℕ}{\ensuremath{\bN}}
\newunicodechar{ℙ}{\ensuremath{\bP}}
\newunicodechar{ℚ}{\ensuremath{\bQ}}
\newunicodechar{ℝ}{\ensuremath{\bR}}
\newunicodechar{𝕏}{\ensuremath{\bX}}
\newunicodechar{ℤ}{\ensuremath{\bZ}}
\newunicodechar{𝒜}{\ensuremath{\sA}}
\newunicodechar{ℬ}{\ensuremath{\sB}}
\newunicodechar{𝒞}{\ensuremath{\sC}}
\newunicodechar{𝒟}{\ensuremath{\sD}}
\newunicodechar{ℰ}{\ensuremath{\sE}}
\newunicodechar{ℱ}{\ensuremath{\sF}}
\newunicodechar{𝒢}{\ensuremath{\sG}}
\newunicodechar{ℋ}{\ensuremath{\sH}}
\newunicodechar{𝒥}{\ensuremath{\sJ}}
\newunicodechar{ℒ}{\ensuremath{\sL}}
\newunicodechar{𝒪}{\ensuremath{\sO}}
\newunicodechar{𝒬}{\ensuremath{\sQ}}
\newunicodechar{𝒯}{\ensuremath{\sT}}
\newunicodechar{𝒲}{\ensuremath{\sW}}

\newunicodechar{∂}{\ensuremath{\partial}}
\newunicodechar{∇}{\ensuremath{\nabla}}

\newunicodechar{↺}{\ensuremath{\circlearrowleft}}
\newunicodechar{∞}{\ensuremath{\infty}}
\newunicodechar{⊕}{\ensuremath{\oplus}}
\newunicodechar{⊗}{\ensuremath{\otimes}}
\newunicodechar{•}{\ensuremath{\bullet}}
\newunicodechar{Λ}{\ensuremath{\wedge}}
\newunicodechar{↪}{\ensuremath{\into}}
\newunicodechar{→}{\ensuremath{\to}}
\newunicodechar{↦}{\ensuremath{\mapsto}}
\newunicodechar{⨯}{\ensuremath{\times}}
\newunicodechar{∪}{\ensuremath{\cup}}
\newunicodechar{∩}{\ensuremath{\cap}}
\newunicodechar{⊋}{\ensuremath{\supsetneq}}
\newunicodechar{⊇}{\ensuremath{\supseteq}}
\newunicodechar{⊃}{\ensuremath{\supset}}
\newunicodechar{⊊}{\ensuremath{\subsetneq}}
\newunicodechar{⊆}{\ensuremath{\subseteq}}
\newunicodechar{⊂}{\ensuremath{\subset}}
\newunicodechar{≥}{\ensuremath{\geq}}
\newunicodechar{≠}{\ensuremath{\neq}}
\newunicodechar{≫}{\ensuremath{\gg}}
\newunicodechar{≪}{\ensuremath{\ll}}

\newunicodechar{≤}{\ensuremath{\leq}}
\newunicodechar{∈}{\ensuremath{\in}}
\newunicodechar{∉}{\ensuremath{\not \in}}
\newunicodechar{∖}{\ensuremath{\setminus}}
\newunicodechar{◦}{\ensuremath{\circ}}
\newunicodechar{°}{\ensuremath{^\circ}}
\newunicodechar{…}{\ifmmode\mathellipsis\else\textellipsis\fi}
\newunicodechar{·}{\ensuremath{\cdot}}
\newunicodechar{⋯}{\ensuremath{\cdots}}
\newunicodechar{∅}{\ensuremath{\emptyset}}
\newunicodechar{⇒}{\ensuremath{\Rightarrow}}

\newunicodechar{⁰}{\ensuremath{^0}}
\newunicodechar{¹}{\ensuremath{^1}}
\newunicodechar{²}{\ensuremath{^2}}
\newunicodechar{³}{\ensuremath{^3}}
\newunicodechar{⁴}{\ensuremath{^4}}
\newunicodechar{⁵}{\ensuremath{^5}}
\newunicodechar{⁶}{\ensuremath{^6}}
\newunicodechar{⁷}{\ensuremath{^7}}
\newunicodechar{⁸}{\ensuremath{^8}}
\newunicodechar{⁹}{\ensuremath{^9}}
\newunicodechar{ⁱ}{\ensuremath{^i}}

\newunicodechar{⌈}{\ensuremath{\lceil}}
\newunicodechar{⌉}{\ensuremath{\rceil}}
\newunicodechar{⌊}{\ensuremath{\lfloor}}
\newunicodechar{⌋}{\ensuremath{\rfloor}}

\newunicodechar{≅}{\ensuremath{\cong}}
\newunicodechar{⇔}{\ensuremath{\Leftrightarrow}}

\DeclareFontFamily{OMS}{rsfs}{\skewchar\font'60}
\DeclareFontShape{OMS}{rsfs}{m}{n}{<-5>rsfs5 <5-7>rsfs7 <7->rsfs10 }{}
\DeclareSymbolFont{rsfs}{OMS}{rsfs}{m}{n}
\DeclareSymbolFontAlphabet{\scr}{rsfs}
\DeclareSymbolFontAlphabet{\scr}{rsfs}

\DeclareFontFamily{U}{mathx}{\hyphenchar\font45}
\DeclareFontShape{U}{mathx}{m}{n}{
      <5> <6> <7> <8> <9> <10>
      <10.95> <12> <14.4> <17.28> <20.74> <24.88>
      mathx10
      }{}
\DeclareSymbolFont{mathx}{U}{mathx}{m}{n}
\DeclareFontSubstitution{U}{mathx}{m}{n}
\DeclareMathAccent{\wcheck}{0}{mathx}{"71}

\DeclareMathOperator{\Aut}{Aut}

\DeclareMathOperator{\Id}{Id}

\DeclareMathOperator{\red}{red}

\DeclareMathOperator{\Spec}{Spec}

\DeclareMathOperator{\supp}{supp}

\newcommand{\sA}{\scr{A}}
\newcommand{\sB}{\scr{B}}
\newcommand{\sC}{\scr{C}}
\newcommand{\sD}{\scr{D}}
\newcommand{\sE}{\scr{E}}
\newcommand{\sF}{\scr{F}}
\newcommand{\sG}{\scr{G}}
\newcommand{\sH}{\scr{H}}

\newcommand{\sJ}{\scr{J}}

\newcommand{\sL}{\scr{L}}

\newcommand{\sO}{\scr{O}}

\newcommand{\sQ}{\scr{Q}}

\newcommand{\sT}{\scr{T}}

\newcommand{\sW}{\scr{W}}

\newcommand{\bA}{\mathbb{A}}
\newcommand{\bB}{\mathbb{B}}
\newcommand{\bC}{\mathbb{C}}
\newcommand{\bD}{\mathbb{D}}
\newcommand{\bE}{\mathbb{E}}
\newcommand{\bF}{\mathbb{F}}

\newcommand{\bN}{\mathbb{N}}

\newcommand{\bP}{\mathbb{P}}
\newcommand{\bQ}{\mathbb{Q}}
\newcommand{\bR}{\mathbb{R}}

\newcommand{\bX}{\mathbb{X}}

\newcommand{\bZ}{\mathbb{Z}}

\theoremstyle{plain}
\newtheorem{thm}{Theorem}[section]

\newtheorem{cor}[thm]{Corollary}
\newtheorem{defn}[thm]{Definition}

\newtheorem{lem}[thm]{Lemma}

\newtheorem{prop}[thm]{Proposition}

\theoremstyle{remark}

\newtheorem{claim}[thm]{Claim}
\newtheorem{c-n-d}[thm]{Claim and Definition}

\newtheorem{notation}[thm]{Notation}
\newtheorem{obs}[thm]{Observation}
\newtheorem{rem}[thm]{Remark}

\newtheorem*{rem-nonumber}{Remark}
\newtheorem{setting}[thm]{Setting}

\numberwithin{equation}{thm}

\setlist[enumerate]{label=(\thethm.\arabic*), before={\setcounter{enumi}{\value{equation}}}, after={\setcounter{equation}{\value{enumi}}}}

\newcommand{\into}{\hookrightarrow}

\newcommand{\wtilde}{\widetilde}
\newcommand{\what}{\widehat}

\newcommand\CounterStep{\addtocounter{thm}{1}\setcounter{equation}{0}}

\newcommand{\factor}[2]{\left. \raise 2pt\hbox{$#1$} \right/\hskip -2pt\raise -2pt\hbox{$#2$}}
\newcommand{\Preprint}[1]{}

\newcommand{\subversionInfo}{}
\newcommand{\svnid}[1]{}
\newcommand{\approvals}[2][Approval]{}

\renewcommand{\phi}{\varphi}

\theoremstyle{remark}
\newtheorem{ass}[thm]{Assumption}
\newtheorem*{com*}{Commentary}

\author{Stefan Kebekus} %
\ifdefined\subtitle
\address{Mathematisches Institut\\
  Albert-Ludwigs-Universität Freiburg\\
  Ernst-Zermelo-Straße 1\\
  79104 Freiburg im Breisgau, Germany \\
  \ \\
  Freiburg Institute for Advanced Studies (FRIAS)\\
  Freiburg im Breisgau, Germany
} %
\else
\address{Mathematisches Institut, Albert-Ludwigs-Universität Freiburg,
  Ernst-Zermelo-Straße 1, 79104 Freiburg im Breisgau, Germany \& Freiburg
  Institute for Advanced Studies (FRIAS), Freiburg im Breisgau, Germany} %
\fi
\email{\href{mailto:stefan.kebekus@math.uni-freiburg.de}{stefan.kebekus@math.uni-freiburg.de}}
\urladdr{\url{https://cplx.vm.uni-freiburg.de}} %
\thanks{Stefan Kebekus gratefully acknowledges support through a fellowship of
  the Freiburg Institute of Advanced Studies (FRIAS)}

\title{Boundedness results for singular Fano varieties, and applications to Cremona groups} %

\ifdefined\subtitle
\subtitle{following Birkar and Prokhorov-Shramov}
\fi

\makeatletter
\hypersetup{
  pdfauthor={\authors},
  pdftitle={\@title},
  pdfsubject={\@subjclass},
  pdfkeywords={\@keywords},
  pdfstartview={Fit},
  pdfpagelayout={TwoColumnRight},
  pdfpagemode={UseOutlines},
  bookmarks,
  colorlinks,
  linkcolor=linkblue,
  citecolor=linkred,
  urlcolor=linkred
}
\makeatother

\DeclareMathOperator{\Bir}{Bir}
\DeclareMathOperator{\DCC}{DCC}
\DeclareMathOperator{\Div}{Div}
\DeclareMathOperator{\Fano}{Fano}
\DeclareMathOperator{\GL}{GL}
\DeclareMathOperator{\Jordan}{Jordan}
\DeclareMathOperator{\lct}{lct}
\DeclareMathOperator{\mult}{mult}
\DeclareMathOperator{\PGL}{ℙ\operatorname{GL}}
\DeclareMathOperator{\vol}{vol}

\newcommand{\cG}{\mathcal G}
\newcommand{\cP}{\mathcal P}
\newcommand{\cR}{\mathcal R}
\newcommand{\cS}{\mathcal S}

\begin{document}

\maketitle

\ifdefined\subtitle
\else
\begin{abstract}
  This survey paper reports on work of Birkar, who confirmed a long-standing
  conjecture of Alexeev and Borisov-Borisov: Fano varieties with mild
  singularities form a bounded family once their dimension is fixed.  Following
  Prokhorov-Shramov, we explain how this boundedness result implies that
  birational automorphism groups of projective spaces satisfy the Jordan
  property, answering a question of Serre in the positive.
\end{abstract}
\fi

\tableofcontents

%
%
\svnid{$Id: 01-intro.tex 85 2019-04-08 23:28:39Z kebekus $}

\section{Main results}
\subversionInfo

Throughout this paper, we work over the field of complex numbers.

\subsection{Boundedness of singular Fano varieties}

A normal, projective variety $X$ is called \emph{Fano} if a negative multiple of
its canonical divisor class is Cartier and if the associated line bundle is
ample.  Fano varieties appear throughout geometry and have been studied
intensely, in many contexts.  For the purposes of this talk, we remark that
Fanos with sufficiently mild singularities constitute one of the fundamental
variety classes in birational geometry.  In fact, given any projective manifold
$X$, the Minimal Model Programme (MMP) predicts the existence of a sequence of
rather special birational transformations, known as ``divisorial contractions''
and ``flips'', as follows,
$$
\xymatrix{ %
  X = X^{(0)} \ar@{-->}[rr]^{α^{(1)}}_{\text{birational}} && X^{(1)}
  \ar@{-->}[rr]^{α^{(2)}}_{\text{birational}} && ⋯
  \ar@{-->}[rr]^{α^{(n)}}_{\text{birational}} && X^{(n)}.
}
$$
The resulting variety $X^{(n)}$ is either canonically polarised (which is to say
that a suitable power of its canonical sheaf is ample), or it has the structure
of a fibre space whose general fibres are either Fano or have numerically
trivial canonical class.  The study of (families of) Fano varieties is thus one
of the most fundamental problems in birational geometry.

\begin{rem}[Singularities]
  Even though the starting variety $X$ is a manifold by assumption, it is well
  understood that we cannot expect the varieties $X^{(•)}$ to be smooth.
  Instead, they exhibit mild singularities, known as ``terminal'' or
  ``canonical'' --- we refer the reader to \cite[Sect.~2.3]{KM98} or
  \cite[Sect.~2]{MR3057950} for a discussion and for references.  If $X^{(n)}$
  admits the structure of a fibre space, its general fibres will also have
  terminal or canonical singularities.  Even if one is primarily interested in
  the geometry of \emph{manifolds}, it is therefore necessary to include
  families of \emph{singular} Fanos in the discussion.
\end{rem}

In a series of two fundamental papers, \cite{Bir16a, Bir16b}, Birkar confirmed a
long-standing conjecture of Alexeev and Borisov-Borisov, \cite{MR1298994,
  MR1166957}, asserting that for every $d ∈ ℕ$, the family of $d$-dimensional
Fano varieties with terminal singularities is bounded: there exists a proper
morphism of quasi-projective schemes over the complex numbers, $u : 𝕏 → Y$, and
for every $d$-dimensional Fano $X$ with terminal singularities a closed point
$y ∈ Y$ such that $X$ is isomorphic to the fibre $𝕏_y$.  In fact, a much more
general statement holds true.

\begin{thm}[\protect{Boundedness of $ε$-lc Fanos, \cite[Thm.~1.1]{Bir16b}}]\label{thm:BAB}
  Given $d ∈ ℕ$ and $ε ∈ ℝ^+$, let $\mathcal X_{d,ε}$ be the family of
  projective varieties $X$ with dimension $\dim_ℂ X = d$ that admit an
  $ℝ$-divisor $B ∈ ℝ\Div(X)$ such that the following holds true.
  \begin{enumerate}
  \item\label{il:BAB1} The tuple $(X,B)$ forms a pair.  In other words: $X$ is
    normal, the coefficients of $B$ are contained in the interval $[0,1]$ and
    $K_X+B$ is $ℝ$-Cartier.
    
  \item\label{il:BAB2} The pair $(X,B)$ is $ε$-lc.  In other words, the total
    log discrepancy of $(X,B)$ is greater than or equal to $ε$.
    
  \item\label{il:BAB3} The $ℝ$-Cartier divisor $-(K_X+B)$ is nef and big.
  \end{enumerate}
  Then, the family $\mathcal X_{d,ε}$ is bounded.
\end{thm}

\begin{rem}[Terminal singularities]
  If $X$ has terminal singularities, then $(X,0)$ is $1$-lc.  We refer to
  Section~\ref{sec:singofpairs}, to Birkar's original papers, or to
  \cite[Sect.~3.1]{MR3224718} for the relevant definitions concerning more
  general classes of singularities.
\end{rem}

For his proof of the boundedness of Fano varieties and for his contributions to
the Minimal Model Programme, Caucher Birkar was awarded with the Fields Medal at
the ICM 2018 in Rio de Janeiro.

\subsubsection{Where does boundedness come from?}
\label{ssec:1-1-1v2}

The brief answer is: ``From boundedness of volumes!'' In fact, if
$(X_t, A_t)_{t ∈ T}$ is a family of tuples where the $X_t$ are normal,
projective varieties of fixed dimension $d$ and $A_t ∈ \Div(X_t)$ are very
ample, and if there exists a number $v ∈ ℕ$ such that
$$
\vol(A_t) := \limsup_{n→∞} \frac{d!·h⁰\bigl( X_t,\, 𝒪_{X_t}(n·A_t) \bigr)}{n^d}
< v
$$
for all $t ∈ T$, then elementary arguments using Hilbert schemes show that the
family $(X_t, A_t)_{t ∈ T}$ is bounded.

For the application that we have in mind, the varieties $X_t$ are the Fano
varieties whose boundedness we would like to show and the divisors $A_t$ will be
chosen as fixed multiples of their anticanonical classes.  To obtain boundedness
results in this setting, Birkar needs to show that there exists one number $m$
that makes all $A_t := -m·K_{X_t}$ very ample, or (more modestly) ensures that
the linear systems $|-m·K_{X_t}|$ define birational maps.  Volume bounds for
these divisors need to be established, and the singularities of the linear
systems need to be controlled.

\subsubsection{Earlier results, related results}
\label{ssec:1-1-2v2}

Boundedness results have a long history, which we cannot cover with any pretence
of completeness.  Boundedness of smooth Fano surfaces and threefolds follows
from their classification.  Boundedness of Fano \emph{manifolds} of arbitrary
dimension was shown in the early 1990s, in an influential paper of Kollár,
Miyaoka and Mori, \cite{KMM92}, by studying their geometry as rationally
connected manifolds.  Around the same time, Borisov-Borisov were able to handle
the toric case using combinatorial methods, \cite{MR1166957}.  For (singular)
surfaces, Theorem~\ref{thm:BAB} is due to Alexeev, \cite{MR1298994}.

Among the newer results, we will only mention the work of Hacon-McKernan-Xu.
Using methods that are similar to those discussed here, but without the results
on ``boundedness of complements'' ($→$ Section~\ref{sec:bcomp}), they are able
to bound the volumes of klt pairs $(X, Δ)$, where $X$ is projective of fixed
dimension, $K_X + Δ$ is numerically trivial and the coefficients of $Δ$ come
from a fixed DCC set, \cite[Thm.~B]{MR3224718}.  Boundedness of Fanos with klt
singularities and fixed Cartier index follows, \cite[Cor.~1.8]{MR3224718}.  In a
subsequent paper \cite{MR3507257} these results are extended to give the
boundedness result that we quote in Theorem~\ref{thm:boundednessCriterion}, and
that Birkar builds on.  We conclude with a reference to \cite{Jiang17, Chen18}
for current results involving $K$-stability and $α$-invariants.  The surveys
\cite{MR2827803, MR3821154} give a more complete overview.

\subsubsection{Positive characteristic}

Apart from the above-mentioned results of Alexeev, \cite{MR1298994}, which hold
over algebraically closed field of arbitrary characteristic, little is known in
case where the characteristic of the base field is positive.

\subsection{Applications}
\label{ssec:1-2}

As we will see in Section~\ref{sec:jordan} below, boundedness of Fanos can be
used to prove the existence of fixed points for actions of finite groups on
Fanos, or more generally rationally connected varieties.  Recall that a variety
$X$ is \emph{rationally connected} if every two points are connected by an
irreducible, rational curve contained in $X$.  This allows us to apply
Theorem~\ref{thm:BAB} in the study of finite subgroups of birational
automorphism groups.

\subsubsection{The Jordan property of Cremona groups}

Even before Theorem~\ref{thm:BAB} was known, it had been realised by Prokhorov
and Shramov, \cite{MR3483470}, that boundedness of Fano varieties with terminal
singularities would imply that the birational automorphism groups of projective
spaces (= Cremona groups, $\Bir(ℙ^d)$) satisfy the \emph{Jordan property}.
Recall that a group $Γ$ is said to \emph{have the Jordan property} if there
exists a number $j ∈ ℕ$ such that every finite subgroup $G ⊂ Γ$ contains a
normal, Abelian subgroup $A ⊂ G$ of index $|G:A| ≤ j$.  In fact, a stronger
result holds.

\begin{thm}[\protect{Jordan property of Cremona groups, \cite[Cor.~1.3]{Bir16b}, \cite[Thm.~1.8]{MR3483470}}]\label{thm:jordan}
  Given any number $d ∈ ℕ$, there exists $j ∈ ℕ$ such that for every complex,
  projective, rationally connected variety $X$ of dimension $\dim_{ℂ} X = d$,
  every finite subgroup $G ⊂ \Bir(X)$ contains a normal, Abelian subgroup
  $A ⊆ G$ of index $|G:A| ≤ j$.
\end{thm}

\begin{rem}
  Theorem~\ref{thm:jordan} answers a question of Serre \cite[6.1]{MR2567402} in
  the positive.  A more detailed analysis establishes the Jordan property more
  generally for all varieties of vanishing irregularity,
  \cite[Thm.~1.8]{MR3292293}.
\end{rem}

Theorem~\ref{thm:jordan} ties in with the general philosophy that finite
subgroups of $\Bir(ℙ^d)$ should in many ways be similar to finite linear groups,
where the property has been established by Jordan more then a century ago.

\begin{thm}[\protect{Jordan property of linear groups, \cite{Jordan1877}}]\label{thm:jordan-lin}
  Given any number $d ∈ ℕ$, there exists $j^{\Jordan}_d ∈ ℕ$ such that every
  finite subgroup $G ⊂ \GL_d(ℂ)$ contains a normal, Abelian subgroup $A ⊆ G$ of
  index $|G:A| ≤ j^{\Jordan}_d$.  \qed
\end{thm}

\begin{rem}[Related results]
  For further information on Cremona groups and their subgroups, we refer the
  reader to the surveys \cite{MR3229352, MR3821147} and to the recent research
  paper \cite{Popov18}.  For the maximally connected components of automorphism
  groups of projective varieties (rather than the full group of birational
  automorphisms), the Jordan property has recently been established by Meng and
  Zhang without any assumption on the nature of the varieties,
  \cite[Thm.~1.4]{MZ18}; their proof uses group-theoretic methods rather than
  birational geometry.  For related results (also in positive characteristic),
  see \cite{Hu18, MR3830471, SV18} and references there.
\end{rem}

\subsubsection{Boundedness of finite subgroups in birational transformation groups}

Following similar lines of thought, Prokhorov and Shramov also deduce
boundedness of finite subgroups in birational transformation groups, for
arbitrary varieties defined over a finite field extension of $ℚ$.

\begin{thm}[\protect{Bounds for finite groups of birational transformation, \cite[Thm.~1.4]{MR3292293}}]
  Let $k$ be a finitely generated field over $ℚ$.  Let $X$ be a variety over
  $k$, and let $\Bir(X)$ denote the group of birational automorphisms of $X$
  over $\Spec k$.  Then, there exists $b ∈ ℕ$ such that any finite subgroup
  $G ⊂ \Bir(X)$ has order $|G| ≤ b$.
\end{thm}

As an immediate corollary, they answer another question of
Serre\footnote{Unpublished problem list from the workshop ``Subgroups of Cremona
  groups: classification'', 29–30 March 2010, ICMS, Edinburgh.  Available at
  \url{http://www.mi.ras.ru/~prokhoro/preprints/edi.pdf}.  Serre's question is
  found on page 7.}, pertaining to finite subgroups in the automorphism group of
a field.

\begin{cor}[\protect{Boundedness for finite groups of field automorphisms, \cite[Cor.~1.5]{MR3292293}}]
  Let $k$ be a finitely generated field over $ℚ$.  Then, there exists
  $b ∈ ℕ$ such that any finite subgroup $G ⊂ \Aut(k)$ has order
  $|G| ≤ b$.
\end{cor}

\subsubsection{Boundedness of links, quotients of the Cremona group}

Birkar's result has further applications within birational geometry.  Combined
with work of Choi-Shokurov, it implies the boundedness of Sarkisov links in any
given dimension, cf.~\cite[Cor.~7.1]{MR2784026}.  In \cite{BLZ},
Blanc-Lamy-Zimmermann use Birkar's result to prove the existence of many
quotients of the Cremona groups of dimension three or more. In particular, they
show that these groups are not perfect and thus not simple.

\subsection{Outline of this paper}

Paraphrasing \cite[p.~6]{Bir16a}, the main tools used in Birkar's work include
the Minimal Model Programme \cite{KM98, BCHM10}, the theory of complements
\cite{MR1892905, MR2448282, MR1794169}, the technique of creating families of
non-klt centres using volumes \cite{MR3224718, MR3034294} and
\cite[Sect.~6]{KollarSingsOfPairs}, and the theory of generalised polarised
pairs \cite{MR3502099}.  In fact, given the scope and difficulty of Birkar's
work, and given the large number of technical concepts involved, it does not
seem realistic to give more than a panoramic presentation of Birkar's proof
here.  Largely ignoring all technicalities,
Sections~\ref{sec:bcomp}--\ref{sec:lcthres} highlight four core results, each of
independent interest.  We explain the statements in brief, sketch some ideas of
proof and indicate how the results might fit together to give the desired
boundedness result.  Finally, Section~\ref{sec:jordan} discusses the application
to the Jordan property in some detail.

\subsection{Acknowledgements}

The author would like to thank Florin Ambro, Serge Cantat, Enrica Floris,
Christopher Hacon, Vladimir Lazić, Benjamin McDonnell, Vladimir Popov, Thomas
Preu, Yuri Prokhorov, Vyacheslav Shokurov, Chenyang Xu and one anonymous reader,
who answered my questions and/or suggested improvements.  Yanning Xu was kind
enough to visit Freiburg and patiently explain large parts of the material to
me.  He helped me out more than just once.  His paper \cite{YXu18}, which
summarises Birkar's results, has been helpful in preparing these notes.  Even
though our point of view is perhaps a little different, it goes without saying
that this paper has substantial overlap with Birkar's own survey \cite{Bir18}.

%
%
\svnid{$Id: 02-notation.tex 79 2019-01-07 13:21:51Z kebekus $}

\section{Notation, standard facts and known results}
\subversionInfo

\subsection{Varieties, divisors and pairs}

We follow standard conventions concerning varieties, divisors and pairs.  In
particular, the following notation will be used.

\begin{defn}[Round-up, round-down and fractional part]
  If $X$ is a normal, quasi-projective variety and $B ∈ ℝ\Div(X)$ an
  $ℝ$-divisor on $X$, we write $⌊ B ⌋$, $⌈ B ⌉$ for the
  round-down and round-up of $B$, respectively.  The divisor
  $\{B\} := B - ⌊ B ⌋$ is called \emph{fractional part of $B$}.
\end{defn}

\begin{defn}[Pair]
  A \emph{pair} is a tuple $(X, B)$ consisting of a normal, quasi-projective
  variety $X$ and an effective $ℝ$-divisor $B$ such that $K_X + B$ is
  $ℝ$-Cartier.
\end{defn}

\begin{defn}[Couple]
  A \emph{couple} is a tuple $(X, B)$ consisting of a normal, projective variety
  $X$ and a divisor $B ∈ \Div(X)$ whose coefficients are all equal to one.
  The couple is called \emph{log-smooth} if $X$ is smooth and if $B$ has simple
  normal crossings support.
\end{defn}

\subsection{$ℝ$-divisors}

While divisors with real coefficients had sporadically appeared in birational
geometry for a long time, the importance of allowing real (rather than rational)
coefficients was highlighted in the seminal paper \cite{BCHM10}, where
continuity- and compactness arguments for spaces of divisors were used in an
essential manner.  Almost all standard definitions for divisors have analogues
for $ℝ$-divisors, but the generalised definitions are perhaps not always
obvious.  For the reader's convenience, we recall a few of the more important
notions here.

\begin{defn}[Big $ℝ$-divisors]
  Let $X$ be a normal, projective variety.  A divisor $B ∈ ℝ\Div(X)$, which need
  not be $ℝ$-Cartier, is called \emph{big} if there exists an an ample
  $H ∈ ℝ\Div(X)$, and effective $D ∈ ℝ\Div(X)$ and an $ℝ$-linear equivalence
  $B \sim_ℝ H + D$.
\end{defn}

\begin{defn}[Volume of an $ℝ$-divisor]
  Let $X$ be a normal, projective variety of dimension $d$.  The \emph{volume}
  of an $ℝ$-divisor $D ∈ ℝ\Div(X)$ is defined as
  $$
  \vol(D) := \limsup_{m→∞} \frac{d!·h⁰\bigl( X,\, 𝒪_X(⌊mD⌋) \bigr)}{m^d}.
  $$
\end{defn}

\begin{defn}[Linear system]
  Let $X$ be a normal, quasi-projective variety and let $M ∈ ℝ\Div(X)$.  The
  \emph{$ℝ$-linear system} $|M|$ is defined as
  $$
  |M|_ℝ := \{ D ∈ ℝ\Div(X) \,|\, D \text{ is effective and } D \sim_{ℝ} M \}.
  $$
\end{defn}

\subsection{Invariants of varieties and pairs}
\label{sec:singofpairs}

We briefly recall a number of standard definitions concerning singularities.  In
brief, if $X$ is smooth, and if $π : \wtilde{X} → X$ is any birational morphism,
where $\wtilde{X}$ it smooth, then any top-form $σ ∈ H⁰ \bigl( X,\, ω_X \bigr)$
pulls back to a holomorphic differential form
$τ ∈ H⁰ \bigl( \wtilde{X},\, ω_{\wtilde{X}} \bigr)$, with zeros along the
positive-dimensional fibres of $π$.  However, if $X$ is singular, if
$π : \wtilde{X} → X$ is a resolution of singularities and if
$σ ∈ H⁰ \bigl( X,\, ω_X \bigr)$ is any section in the (pre-)dualising sheaf,
then the pull-back of $σ$ will only be a rational differential form on
$\wtilde{X}$ which might well have poles along the positive-dimensional fibres
of $π$.  The idea in the definition of ``log discrepancy'' is to use this pole
order to measure the ``badness'' of the singularities on $X$.  We refer the
reader to one of the standard references \cite[Sect.~2.3]{KM98} and
\cite[Sect.~2]{MR3057950} for an-depth discussion of these ideas and of the
singularities of the Minimal Model Programme.  Since the notation is not uniform
across the literature\footnote{The papers \cite{Bir16a, Bir16b, BCHM10} denote
  the log discrepancy by $a(D, X, B)$, while the standard reference books
  \cite{KM98, MR3057950} write $a(D, X, B)$ for the standard (= ``non-log'')
  discrepancies.}, we spend a few lines to fix notation and briefly recall the
central definitions of the field.

\begin{defn}[Log discrepancy]\label{not:logdiscrep}
  Let $(X,B)$ a pair and let $π : \wtilde{X} → X$ be a log resolution of
  singularities, with exceptional divisors $(E_i)_{1 ≤ i ≤ n}$.  Since $K_X+B$
  is $ℝ$-Cartier by assumptions, there exists a well-defined notion of
  pull-back, and a unique divisor $B_{\wtilde{X}} ∈ ℝ\Div(\wtilde{X})$ such that
  $K_{\wtilde{X}} + B_{\wtilde{X}} = π^* (K_X+B)$ in $ℝ\Div(\wtilde{X})$.  If
  $D$ is any prime divisor on~$\wtilde{X}$, we consider the \emph{log
    discrepancy}
  $$
  a_{\log}(D, X, B) := 1 - \mult_D B_{\wtilde{X}}.
  $$
  The infimum over all such numbers,
  $$
  a_{\log}(X, B) :=
  \inf \{ a_{\log}(D, X, B) \:|\: π:\wtilde{X} → X \text{ a log resolution and } D ∈ \Div(\wtilde{X}) \text{ prime}\}
  $$
  is called the \emph{total log discrepancy of the pair $(X,B)$}.
\end{defn}

The total log discrepancy measures how bad the singularities are: the smaller
$a_{\log}(X, B)$ is, the worse the singularities are.  Table~\vref{tab:x1} lists
the classes of singularities will be relevant in the sequel.  In addition,
$(X, B)$ is called \emph{plt} if $a_{\log}(D, X, B) > 0$ for every resolution
$π : \wtilde{X} → X$ and every \emph{exceptional} divisor $D$ on~$\wtilde{X}$.
The class of $ε$-lc singularities, which is perhaps the most relevant for our
purposes, was introduced by Alexeev.

\begin{table}
  \centering
  \begin{tabular}{ccc}
    \rowcolor{gray!20} If …, then & & $(X,B)$ is called ``…'' \\ \hline
    $a_{\log}(X, B) ≥ 0$ & … & \emph{log canonical (or ``lc'')} \\
    $a_{\log}(X, B) > 0$ & … & \emph{Kawamata log terminal (or ``klt'')} \\
    $a_{\log}(X, B) ≥ ε$ & … & \emph{$ε$-log canonical (or ``$ε$-lc'')} \\
    $a_{\log}(X, B) ≥ 1$ & … & \emph{canonical} \\
    $a_{\log}(X, B) > 1$ & … & \emph{terminal}
  \end{tabular}
  \bigskip
  \caption{Singularities of the Minimal Model Programme}
  \label{tab:x1}
\end{table}

\subsubsection{Places and centres}

The divisors $D$ that appear in the definition log discrepancy deserve special
attention, in particular if $a_{\log}(D, X, B) ≤ 0$.

\begin{defn}[Non-klt places and centres]
  Let $(X,B)$ a pair.  A \emph{non-klt place} of $(X, B)$ is a prime divisor $D$
  on birational models of $X$ such that $a_{\log}(D, X, B) ≤ 0$.  A
  \emph{non-klt centre} is the image on $X$ of a non-klt place.  When $(X, B)$ is
  lc, a non-klt centre is also called an \emph{lc centre}.
\end{defn}

\subsubsection{Thresholds}

Suppose that $(X,B)$ is a klt pair, and that $D$ is an effective divisor on $X$.
The pair $(X, B+t·D)$ will then be log-canonical for sufficiently small numbers
$t$, but cannot be klt when $t$ is large.  The critical value of $t$ is called
the \emph{log-canonical threshold}.

\begin{defn}[\protect{LC threshold, compare \cite[Sect.~9.3.B]{Laz04-II}}]\label{def:lct}
  Let $(X,B)$ be a klt pair.  If $D ∈ ℝ\Div(X)$ is effective, one defines the
  \emph{lc threshold of $D$ with respect to $(X,B)$} as
  \begin{align*}
    \lct \bigl(X,\, B,\, D \bigr) & := \sup \bigl\{ t ∈ ℝ \:\bigl|\: (X,B+t·D) \text{ is lc} \bigr\}.  \\
    \intertext{If $Δ ∈ ℝ\Div(X)$ is $ℝ$-Cartier with non-empty $ℝ$-linear system (but not
    necessarily effective itself), one defines \emph{lc threshold of $|Δ|_{ℝ}$
    with respect to $(X,B)$} as}
    \lct \bigl(X,\, B,\, |Δ|_ℝ \bigr) & := \inf \bigl\{ \lct(X,B,D) \:\bigl|\: D ∈ |Δ|_ℝ \bigr\}.
  \end{align*}
\end{defn}

\begin{rem}\label{rem:sflct}
  In the setting of Definition~\ref{def:lct}, it is a standard fact that
  $$
  \lct \bigl(X,\, B,\, |Δ|_ℝ \bigr) = \sup \bigl\{ t ∈ ℝ \:\bigl|\: (X, B+t·D) \text{ is lc for every } D ∈ |Δ|_{ℝ} \bigr\}.
  $$
  In particular, if $(X,B)$ is klt, then $(X, B+t'·D)$ is lc for every
  $D ∈ |Δ|_{ℝ}$ and every $0 < t' < t$.
\end{rem}

\begin{notation}
  If $B = 0$, we omit it from the notation and write
  $\lct \bigl(X,\, |Δ|_ℝ \bigr)$ and $\lct \bigl(X,\, D \bigr)$ in short.
\end{notation}

\subsection{Fano varieties and pairs}

Fano varieties come in many variants.  For the purposes of this overview, the
following classes of varieties will be the most relevant.

\begin{defn}[\protect{Fano and weak log Fano pairs, \cite[Sect.~2.10]{Bir16a}}]
  ~
  \begin{itemize}
  \item A projective pair $(X, B)$ is called \emph{log Fano} if $(X, B)$ is lc
    and if $-(K_X+B)$ is ample.  If $B = 0$, we just say that $X$ is Fano.

  \item A projective pair $(X, B)$ is called is called \emph{weak log Fano} if
    $(X, B)$ is lc and $-(K_X + B)$ is nef and big.  If $B = 0$, we just say
    that $X$ is weak Fano.
  \end{itemize}
\end{defn}

\begin{rem}[Relative notions]
  There exist relative versions of the notions discussed above.  If $(X, B)$ is
  any quasi-projective pair, if $Z$ is normal and if $X → Z$ is surjective,
  projective and with connected fibres, we say $(X, B)$ is log Fano over $Z$ if
  it is lc and if $-(K_X + B)$ is relatively ample over $Z$.  Ditto with ``weak
  log Fano''.
\end{rem}

\subsection{Varieties of Fano type}

Varieties $X$ that \emph{admit} a boundary $B$ that makes $(X,B)$ a Fano pair
are said to be of \emph{Fano type}.  This notion was introduced by Prokhorov and
Shokurov in \cite{MR2448282}.  We refer to that paper for basic properties of
varieties of Fano type.

\begin{defn}[\protect{Varieties of Fano type, \cite[Lem.\ and Def.~2.6]{MR2448282}}]
  A normal, projective variety $X$ is said to be \emph{of Fano type} if there
  exists an effective, $ℚ$-divisor $B$ such that $(X,B)$ is klt and weak log
  Fano pair.  Equivalently: there exists a big $ℚ$-divisor $B$ such that
  $K_X + B \sim_ℚ 0$ and such that $(X, B)$ is a klt pair.
\end{defn}

\begin{rem}[Varieties of Fano type are Mori dream spaces]\label{rem:MoriDream}
  If $X$ is of Fano type, recall from \cite[Sect.~1.3]{BCHM10} that $X$ is a
  ``Mori dream space''.  Given any $ℝ$-Cartier divisor $D ∈ ℝ\Div(X)$, we can
  then run the $D$-Minimal Model Programme and obtain a sequence of extremal
  contractions and flips, $X \dasharrow Y$.  If the push-forward of $D_Y$ of $Y$
  is nef over, we call $Y$ a minimal model for $D$.  Otherwise, there exists a
  $D_Y$-negative extremal contraction $Y → T$ with $\dim Y > \dim T$, and we
  call $Y$ a Mori fibre space for $D$.
\end{rem}

\begin{rem}[Relative notions]
  As before, there exists an obvious relative version of the notion ``Fano
  type''.  Remark~\ref{rem:MoriDream} generalises to this relative setting.
\end{rem}

Varieties of Fano type come in two flavours that often need to be treated
differently.  The following notion, which we recall for later use, has been
introduced by Shokurov.

\begin{defn}[Exceptional and non-exceptional pairs]\label{def:exceptional}
  Let $(X, B)$ be a projective pair, and assume that there exists an effective
  $P ∈ ℝ\Div(X)$ such that $K_X + B + P \sim_{ℝ} 0$.  We say $(X, B)$ is
  \emph{non-exceptional} if we can choose $P$ so that $(X, B+P)$ is \emph{not}
  klt.  We say that $(X, B)$ is \emph{exceptional} if $(X, B+P)$ is klt for
  every choice of $P$.
\end{defn}

%
%
\svnid{$Id: 03-genlPairs.tex 78 2019-01-07 12:48:44Z kebekus $}

\section{b-Divisors and generalised pairs}
\subversionInfo

In addition to the classical notions for singularities of pairs that we recalled
in Section~\ref{sec:singofpairs} above, much of Birkar's work uses the notion of
\emph{generalised polarised pairs}.  The additional flexibility of this notion
allows for inductive proofs, but adds substantial technical difficulties.
Generalised pairs were introduced by Birkar and Zhang in \cite{MR3502099}.

\subsubsection*{Disclaimer}

The notion of generalised polarised pairs features prominently in Birkar's work,
and should be presented in an adequate manner.  The technical complications
arising from this notion are however substantial and cannot be explained within
a few pages.  As a compromise, this section briefly explains what generalised
pairs are, and how they come about in relevant settings.  Section~\ref{ssec:BC}
pinpoints one place in Birkar's inductive scheme of proof where generalised
pairs appear naturally, and explains why most (if not all) of the material
presented in this survey should in fact be formulated and proven for generalised
pairs.  For the purpose of exposition, we will however ignore this difficulty and
discuss the classical case only.

\subsection{Definition of generalised pairs}

To begin, we only recall a minimal subset of the relevant definitions, and refer
to \cite[Sect.~2]{Bir16a} and to \cite[Sect.~4]{MR3502099} for more details.  We
start with the notion of b-divisors, as introduced by Shokurov in
\cite{MR1420223}, in the simplest case.

\begin{defn}[b-divisor]\label{def:1v2}
  Let $X$ be a variety.  We consider projective, birational morphisms $Y → X$
  from normal varieties $Y$, and for each $Y$ a divisor $M_Y ∈ ℝ\Div(Y)$.  The
  collection $M := (M_Y)_Y$ is called \emph{$b$-divisor} if for any morphism
  $f : Y' → Y$ of birational models over $X$, we have $M_Y = f_* (M_{Y'})$.
\end{defn}

\begin{defn}[b-$ℝ$-Cartier and b-Cartier b-divisors]
  Setting as in Definition~\ref{def:1v2}.  A b-divisor $M$ is called
  \emph{b-$ℝ$-Cartier} if there exists one $Y$ such that $M_Y$ is $ℝ$-Cartier
  and such that for any morphism $f : Y' → Y$ of birational models over $X$, we
  have $M_{Y'} = f^* (M_Y)$.  Ditto for \emph{b-Cartier} b-divisors.
\end{defn}

\begin{defn}[\protect{Generalised polarised pair, \cite[Sect.~2.13]{Bir16a}, \cite[Def.~1.4]{MR3502099}}]\label{def:gpp}
  Let $Z$ be a variety.  A \emph{generalised polarised pair over $Z$} is a tuple
  consisting of the following data:
  \begin{enumerate}
  \item a normal variety $X$ equipped with a projective morphism $X → Z$,
  \item an effective $ℝ$-divisor $B ∈ ℝ\Div(X)$, and
  \item a b-$ℝ$-Cartier b-divisor over $X$ represented as $(φ: X' → X, M')$,
    where $M' ∈ ℝ\Div(X')$ is nef over $Z$, and where $K_X + B + φ_* M'$ is
    $ℝ$-Cartier.
  \end{enumerate}
\end{defn}

\begin{notation}[Generalised polarised pair]
  In the setup of Definition~\ref{def:gpp}, we usually write $M := φ_* M'$ and
  say that $(X, B+M)$ is a generalised pair with data $X' \overset{φ}{→} X → Z$
  and $M'$.  In contexts where $Z$ is not relevant, we usually drop it from the
  notation: in this case one can just assume $X → Z$ is the identity.  When $Z$
  is a point we also drop it but say the pair is projective.
\end{notation}

\begin{obs}
  Following \cite[p.~286]{MR3502099} we remark that Definition~\ref{def:gpp} is
  flexible with respect to $X'$ and $M'$.  To be more precise, if $g : X'' → X'$
  is a projective birational morphism from a normal variety, then there is no
  harm in replacing $X'$ with $X''$ and replacing $M'$ with $g^* M'$.
\end{obs}

\subsection{Singularities of generalised pairs}

All notions introduced in Section~\ref{sec:singofpairs} have analogues in the
setting of generalised pairs.  Again, we cover only the most basic definition
here.

\begin{defn}[Generalised log discrepancy, singularity classes]
  Consider a generalised polarised pair $(X, B+M)$ with data
  $X' \overset{φ}{→} X → Z$ and $M'$, where $φ$ is a log resolution of $(X, B)$.
  Then, there exists a uniquely determined divisor $B'$ on $X'$ such that
  $$
  K_{X'} + B' + M' = φ^* (K_X + B + M)
  $$
  If $D ∈ \Div(X')$ is any prime divisor, the \emph{generalised log discrepancy}
  is defined to be
  $$
  a_{\log}(D, X, B+M) := 1 - \mult_D B'.
  $$
  As before, we define the \emph{generalised total log discrepancy}
  $a_{\log}(X, B+M)$ by taking the infimum over all $D$ and all resolutions.  In
  analogy to the definitions of Table~\ref{tab:x1}, we say that the generalised
  polarised pair is \emph{generalised lc} if $a_{\log}(X, B+M) ≥ 0$.  Ditto for
  all the other definitions.
\end{defn}

\subsection{Example: Fibrations and the canonical bundle formula}
\label{ssec:gpfib}

We discuss a setting where generalised pairs appear naturally.  Let $Y$ be a
normal pair variety, and let $f : Y → X$ be a fibration: the space $X$ is
projective, normal and of positive dimension, the morphism $f$ is surjective
with connected fibres.  Also, assume that $K_Y$ is $ℚ$-linearly equivalent to
zero over $X$, so that there exists $L_X ∈ ℚ\Div(X)$ with $K_Y \sim_ℚ f^* L_X$.
Ideally, one might hope that it would be possible to choose $L_X = K_X$, but
this is almost always wrong --- compare Kodaira's formula for the canonical
bundle of an elliptic fibration, \cite[Sect.~V.12]{HBPV}.  To fix this issue, we
define a first correction term $B ∈ ℚ\Div(X)$ as
$$
B := \sum_{\mathclap{\substack{D ∈ \Div(X)\\\text{prime}}}} (1-t_D)·D \quad
\text{where} \quad t_D := \lct° \bigl(Y,\, Δ_Y,\, f^*D \bigr)
$$
The symbol $\lct°$ denotes a variant of the lc threshold introduced in
Definition~\ref{def:lct}, which measures the singularities of
$\bigl(Y, f^*D \bigr)$ only over the generic point of $D$.  Since $X$ is smooth
in codimension one, this also solves the problem of defining $f^* D$.  Finally,
one chooses $M ∈ ℚ\Div(X)$ such that $K_X + B + M$ is $ℚ$-Cartier and such that
the desired $ℚ$-linear equivalence holds,
$$
K_Y \sim_ℚ f^* ( K_X + B + M ).
$$
The divisor $B$ is usually called the ``discriminant part'' of the correction
term.  It detects singularities of the fibration, such as multiple or otherwise
singular fibres, over codimension one points of $X$.  The divisor $M$ is called
the ``moduli part''.  It is harder to describe.  While we have defined it only
up to $ℚ$-linear equivalence, a more involved construction can be used to define
it as an honest divisor.

\begin{com*}
  Conjecturally, the moduli part carries information on the birational variation
  of the fibres of $f$, \cite{MR1646046}.  We refer to \cite{MR2359346} and to
  the introduction of the recent research paper \cite{FL18} for an overview, but
  see also \cite{MR3329677}.
\end{com*}

\subsubsection{Behaviour under birational modifications}

We ask how the moduli part of the correction term behaves under birational
modification.  To this end, let $φ : X' → X$ be a birational morphism of normal,
projective varieties.  Choosing a resolution $Y'$ of $Y ⨯_X X'$, we find a
diagram as follows,
$$
\xymatrix{%
  Y' \ar[rr]^{Φ \text{, birational}} \ar[d]_{f'} && Y \ar[d]^f \\
  X' \ar[rr]_{φ \text{, birational}} && X.  }
$$
Set $Δ_{Y'} := Φ^* K_Y - K_{Y'}$.  Generalising the definition of $\lct°$ a
little to allow for negative coefficients in $Δ_{Y'}$, one can then define $B'$
similarly to the construction above,
$$
B' := \sum_{\mathclap{\substack{D ∈ \Div(X')\\\text{prime}}}} (1-t'_D)·D \quad
\text{where} \quad t'_D := \lct° \bigl(Y',\, Δ_{Y'},\, (f')^*D \bigr).
$$
Finally, one may then choose $M' ∈ ℚ\Div(X')$ such that
\begin{align*}
  K_{Y'} + Δ_{Y'} & \sim_ℚ (f')^* (K_{X'} + B' + M' ), \\
  K_{X'} + B' + M' & = φ^*(K_X+B+M)
\end{align*}
and $B = φ_* B'$ as well as $M = φ_* M'$.

\subsubsection{Relation to generalised pairs}

Now assume that $Y$ is lc.  The divisor $B$ will then be effective.  However,
much more is true: after passing to a certain birational model $X'$ of $X$, the
divisor $M_{X'}$ is nef and for any higher birational model $X'' → X'$, the
induced $M_{X''}$ on $X''$ is the pullback of $M_{X''}$, \cite{MR1646046,
  MR2153078, MR2359346} and summarised in \cite[Thm.~3.6]{Bir16a}.  In other
words, going to a sufficiently high birational model of $X'$ of $X$, the moduli
parts $M'$ define an b-$ℝ$-Cartier b-divisor.  Moreover, this b-divisor is
b-nef.  We obtain a generalised polarised pair $(X, B+M)$ with data
$X' \overset{φ}{→} X → \Spec ℂ$ and $M'$.  This generalised pair is generalised
lc by definition.

\begin{com*}
  A famous conjecture of Prokhorov and Shokurov \cite[Conj.~7.13]{MR2448282}
  asserts that the moduli divisor $M_{X''}$ is semiample, on any sufficiently
  high birational model $X''$ of $X$.  More precisely, it is expected that a
  number $m$ exists that depends only on the general fibre of $f$ such that all
  divisors $m·M_{X''}$ are basepoint free.  If this conjecture was solved, it is
  conceivable that Birkar’s work could perhaps be rewritten in a manner that
  avoids the notion of generalised pairs.
\end{com*}

\begin{rem}[Outlook]
  The construction outlined in this section is used in the proof of
  ``Boundedness of complements'', as sketched in Section~\ref{ssec:BC} below.
  It generalises fairly directly to pairs $(Y, Δ_Y)$, and even to tuples where
  $Δ_Y$ is not necessarily effective, \cite[Sect.~3.4]{Bir16a}.
\end{rem}

%
%
\svnid{$Id: 04-boundedness-of-complements.tex 82 2019-01-08 08:34:53Z kebekus $}

\section{Boundedness of complements}
\label{sec:bcomp}
\subversionInfo

\subsection{Statement of result}

One of the central concepts in Birkar's papers \cite{Bir16a, Bir16b} is that of
a \emph{complement}.  The notion of a ``complement'' is an ingenious concept of
Shokurov that was introduced in his investigation of threefold flips,
\cite[Sect.~5]{zbMATH00146455}.  We recall the definition in brief.

\begin{defn}[\protect{Complement, \cite[Sect.~2.18]{Bir16a}}]\label{defn:comple}
  Let $(X, B)$ be a projective pair and $m ∈ ℕ$.  An \emph{$m$-complement of
    $K_X + B$} is a $ℚ$-divisor $B^+$ with the following properties.
  \begin{enumerate}
  \item\label{il:c1} The tuple $(X, B^+)$ is an lc pair.
  \item\label{il:c2} The divisor $m·(K_X + B^+)$ is linearly equivalent to $0$.
    In particular, $m·B^+$ is integral.
  \item\label{il:c3} We have $m·B^+ ≥ m·⌊B⌋ +⌊(m+1)·\{B\}⌋$.
  \end{enumerate}
\end{defn}

\begin{rem}[Complements give sections]\label{rem:3-2}
  Setting as in Definition~\ref{defn:comple}.  If $m$ can be chosen such that
  $m·⌊B⌋ +⌊(m+1)·\{B\}⌋ ≥ m·B$, then Item~\ref{il:c2} guarantees that $-m·(K_X+B)$
  is linearly equivalent to the effective divisor $m·(B^+-B)$.  In particular,
  the sheaf $𝒪_X\bigl(-m·(K_X+B)\bigr)$ admits a global section.
\end{rem}

\begin{rem}
  In view of Item~\ref{il:c2}, Shokurov considers complements as divisors that
  make the lc pair $(X, B^+)$ ``Calabi-Yau'', hence ``flat''.
\end{rem}

The following result, which asserts the existence of complements with bounded
$m$, is one of the core results in Birkar's paper \cite{Bir16a}.  A proof of
Theorem~\ref{thm:boundOfCompl} is sketched in Section~\vref{ssec:BC}.

\begin{thm}[\protect{Boundedness of complements, \cite[Thm.~1.7]{Bir16a}}]\label{thm:boundOfCompl}
  Given $d ∈ ℕ$ and a finite set $\mathcal{R} ⊂ [0,1] ∩ ℚ$, there exists $m ∈ ℕ$
  with the following property.  If $(X,B)$ is any log canonical, projective
  pair, where
  \begin{enumerate}
  \item\label{il:r1} $X$ is of Fano type and $\dim X = d$,
  \item the coefficients of $B$ are of the form $\frac{l-r}{l}$, for
    $r ∈ \mathcal{R}$ and $l ∈ ℕ$,
  \item $-(K_X+B)$ is nef,
  \end{enumerate}
  then there exists an $m$-complement $B^+$ of $K_X+B$ that satisfies
  $B^+ ≥ B$.  The divisor $B^+$ is also an $(m·l)$-complement, for every
  $l ∈ ℕ$.
\end{thm}

\begin{rem}[Complements give sections]\label{rem:3-4}
  Given a pair $(X,B)$ as in Theorem~\ref{thm:boundOfCompl} and a number $l ∈ ℕ$
  such that $(ml)·B$ is integral, then $ml·⌊B⌋ +⌊(ml+1)·\{B\}⌋ ≥ ml·B$, and
  Remark~\ref{rem:3-2} implies that
  $h⁰ \bigl( X,\, 𝒪_X(-ml·(K_X+B)) \bigr) > 0$.
\end{rem}

\subsection{Idea of application}

We aim to show Theorem~\ref{thm:BAB}: under suitable assumptions on the
singularities the family of Fano varieties is bounded.  The proof relies on the
following boundedness criterion of Hacon and Xu that we quote without proof (but
see Sections~\ref{ssec:1-1-1v2} and \ref{ssec:1-1-2v2} for a brief discussion).
Recall that a set of numbers is DCC if every strictly descending sequence of
elements eventually terminates.

\begin{thm}[\protect{Boundedness criterion, \cite[Thm.~1.3]{MR3507257}}]\label{thm:boundednessCriterion}
  Given $d ∈ ℕ$ and a $\DCC$ set $I ⊂ [0,1] ∩ ℚ$, let $\mathcal{Y}_{d,I}$ be the
  family of pairs $(X,B)$ such that the following holds true.
  \begin{enumerate}
  \item The pair $(X, B)$ is projective, klt, and of dimension $\dim_ℂ X = d$.
  \item The coefficients of $B$ are contained in $I$.  The divisor $B$ is big
    and $K_X + B \sim_ℚ 0$.
  \end{enumerate}
  Then, the family $\mathcal{Y}_{d,I}$ is bounded.  \qed
\end{thm}

With the boundedness criterion in place, the following observation relates
``boundedness of complements'' to ``boundedness of Fanos'' and explains what
pieces are missing in order to obtain a full proof.

\begin{obs}\label{obs:1}
  Given $d ∈ ℕ$ and $ε ∈ ℝ^+$, Theorem~\ref{thm:boundOfCompl} gives a number
  $m ∈ ℕ$ such that every $ε$-lc Fano variety $X$ with $-K_X$ nef admits an
  effective complement $B^+$ of $K_X = K_X+0$, with coefficients in the set
  $\{\frac{1}{m}, \frac{2}{m}, …, \frac{m}{m}\}$.  If one could in addition
  always choose $B^+$ so that $(X,B^+)$ was klt rather than merely lc, then
  Theorem~\ref{thm:boundednessCriterion} would immediately apply to show that
  the family of $ε$-lc Fano varieties with $-K_X$ nef is bounded.
\end{obs}

As an important step towards boundedness of $ε$-lc Fanos, we will see in
Section~\ref{sec:EB} how the theorem on ``effective birationality'' together
with Theorem~\ref{thm:boundednessCriterion} and Observation~\ref{obs:1} can be
used to find a boundedness criterion (=Proposition~\vref{prop:4.3a}) that
applies to a relevant class of klt, weak Fano varieties.

\subsection{Variants and generalisations}
\label{sec:vandg}

Theorem~\ref{thm:boundOfCompl} is in fact part of a much larger package,
including boundedness of complements in the relative setting,
\cite[Thm.~1.8]{Bir16a}, and boundedness of complements for generalised
polarised pairs, \cite[Thm.~1.10]{Bir16a}.  To keep this survey reasonably short,
we do not discuss these results here, even though they are of independent
interest, and play a role in the proofs of Theorems~\ref{thm:boundOfCompl} and
\ref{thm:BAB}.

\subsection{Idea of proof for Theorem~\ref*{thm:boundOfCompl}}
\label{ssec:BC}

We sketch a proof of ``boundedness of complements'', following
\cite[p.~6ff]{Bir16a} in broad strokes, and filling in some details now and
then.  In essence, the proof works by induction over the dimension, so assume
that $d$ is given and that everything was already shown for varieties of lower
dimension.

\subsubsection*{Simplification}

Theorem~\ref{thm:boundOfCompl} considers a finite set $\cR ⊂ [0,1] ∩ ℚ$, and log
canonical pairs $(X,B)$, where the coefficients of $B$ are contained in the set
$$
Φ(\cR) := \bigl\{ \textstyle{\frac{l-r}{l}} \,|\, r ∈ \mathcal{R} \text{ and
} l ∈ ℕ \bigr\}.
$$
The set $Φ(\cR)$ is infinite, and has $1 ∈ ℚ$ as its only accumulation point.
Birkar shows that it suffices to treat the case where the coefficient set is
finite.  To this end, he constructs in \cite[Prop.~2.49 and
Constr.~6.13]{Bir16a} a number $ε' ≪ 1$ and shows that it suffices to consider
pairs with coefficients in the finite set $Φ(\cR) ∩ [0,1-ε'] ∪ \{1\}$.  In fact,
given any $(X,B)$, he considers the divisor $B'$ obtained by replacing those
coefficients on $B$ that lie in the range $(1-ε',1)$ with $1$.  Next, he
constructs a birational model $(X'', B'')$ of $(X, B')$ that satisfies all
assumptions Theorem~\ref{thm:boundOfCompl}.  His construction guarantees that to
find an $n$-complement for $(X,B)$ it is equivalent to find an $n$-complement
for $(X'', B'')$.  Among other things, the proof involves carefully constructed
runs of the Minimal Model Programme, Hacon-McKernan-Xu's local and global ACC
for log canonical thresholds \cite[Thms.~1.1 and 1.5]{MR3224718}, and the
extension of these results to generalised pairs \cite[Thm.~1.5 and
1.6]{MR3502099}.

\begin{rem}
  Recall from Remark~\ref{rem:MoriDream} that Assumption~\ref{il:r1} (``$X$ is
  of Fano type'') allows us to run Minimal Model Programmes on arbitrary
  divisors.
\end{rem}

Along similar lines, Birkar is able to modify $(X'', B'')$ by further birational
transformation, and eventually proves that it suffices to show boundedness of
complements for pairs that satisfy the following additional assumptions.

\begin{ass}
  The coefficient set of $(X, B)$ is contained in $\cR$ rather than in $Φ(\cR)$,
  and one of the following holds true.
  \begin{enumerate}
  \item\label{il:sc1} The divisor $-(K_X+B)$ is nef and big, and $B$ has a
    component $S$ with coefficient $1$ that is of Fano type.

  \item\label{il:sc2} There exists a fibration $f \colon X → T$ and
    $K_X+B\equiv 0$ along that fibration.

  \item\label{il:sc3} The pair $(X,B)$ is exceptional.
  \end{enumerate}
\end{ass}

\begin{com*}
  The main distinction is between Case~\ref{il:sc3} and Case~\ref{il:sc1}.  In
  fact, if $(X,B)$ is not exceptional, recall from
  Definition~\ref{def:exceptional} that there exists an effective $P ∈ ℝ\Div(X)$
  such that $K_X + B + P \sim_{ℝ} 0$ and such that $(X, B+P)$ is \emph{not} klt.
  This allows us to find a birational model whose boundary contains a divisor
  with multiplicity one.  Case~\ref{il:sc2} comes up if the runs of the Minimal
  Model Programmes used in the construction of birational models terminates with
  a Kodaira fibre space.
\end{com*}

The three cases \ref{il:sc1}--\ref{il:sc3} require very different inductive
treatments.

\subsubsection*{Case~\ref{il:sc1}}

We consider only the simple case where $S = ⌊ B ⌋$ is a normal prime divisor,
where $(X,B)$ is plt near $S$ and where $-(K_X+B)$ is ample.  Setting
$B_S := (K_X+B)|_S - K_S$, the coefficients are contained in a finite set $\cR'$
of rational numbers that depends only on $\cR$ and on $d$.  In summary, the pair
$(S, B_S)$ reproduces the assumptions of Theorem~\ref{thm:boundOfCompl}, and by
induction we obtain a number $n ∈ ℕ$ that depends only on $\cR$ and $d$, such
that
\begin{enumerate}
\item\label{il:i1} the divisor $n·B_S$ is integral, and
\item\label{il:i2} there exists an $n$-complement $B^+_S$ of $K_S+B_S$.
\end{enumerate}
Following \cite[Prop.~6.7]{Bir16a}, we aim to extend $B^+_S$ from $S$ to a
complement $B^+$ of $K_X+B$ on $X$.  As we saw in in Remark~\ref{rem:3-4},
Item~\ref{il:i1} guarantees that $n·(B^+_S-B_S)$ is effective, so that the
complement $B^+_S$ gives rise to a section in
$$
H⁰ \bigl( S,\, n·(B^+_S-B_S) \bigr) = H⁰ \bigl( S,\, -n·(K_S+B_S) \bigr)
$$
But then, looking at the cohomology of the standard ideal sheaf sequence,
$$
H⁰ \bigl( X,\, -n·(K_X+B) \bigr) →
\underbrace{H⁰ \bigl( S,\, -n·(K_X+B)|_S \bigr)}_{\not = 0 \text{ by Rem.~\ref{rem:3-4}}} →
\underbrace{H¹ \bigl( X,\, -n·(K_X+B) - S \bigr) }_{= 0 \text{ by Kawamata-Viehweg vanishing}}
$$
we find that the section extends to $X$ and defines an associated divisor
$B^+ ∈ \lvert-(K_X+B)\rvert_ℚ$.  Using the connectedness principle for non-klt
centres\footnote{For generalised pairs, this is \cite[Lem.~2.14]{Bir16a}}, one
argues that $B^+$ is the desired complement.

\subsubsection*{Case~\ref{il:sc2}}
\label{ssect:sc2}

Given a fibration $f: X → T$, we apply the construction of
Section~\ref{ssec:gpfib}, in order to equip the base variety $T$ with the
structure of a generalised polarised pair $(T, B+M)$, with data
$T' \overset{φ}{→} T → \Spec ℂ$ and $M'$.

Adding to the results explained in Section~\ref{ssec:gpfib}, Birkar shows that
the coefficients of $B$ and $M$ are not arbitrary.  The coefficients of $B$ are
in $Φ(\cS)$ for some fixed finite set $\cS$ of rational numbers that depends
only on $\cR$ and $d$.  Along similar lines, there exists a bounded number
$p ∈ ℕ$ such that $p·M$ is integral.  The plan is now to use induction to find a
bounded complement for $K_T + B + M$ and pull it back to $X$.  This plan works
out well, but requires us to formulate and prove all results pertaining to
boundedness of complements in the setting of generalised polarised pairs.  All
the arguments sketched here continue to work, \emph{mutatis mutandis}, but the
level of technical difficulty increases substantially.

\subsubsection*{Case~\ref{il:sc3}}

There is little that we can say in brief about this case.  Still, assume for
simplicity that $B=0$ and that $X$ is a Fano variety.  If we could show that $X$
belongs to a bounded family, then we would be done.  Actually we need something
weaker: effective birationality.  Assume we have already proved
Theorem~\ref{thm:effBir}.  Then there is a bounded number $m ∈ ℕ$ such that
$|-mK_X|$ defines a birational map.  Pick $M∈ |-mK_X|$ and let
$B^+ := \frac{1}{m}·M$.  Since $X$ is exceptional, $(X, B^+)$ is automatically
klt, hence $K_X+B^+$ is an $m$-complement.

Although this gives some idea of how one may get a bounded complement but in
practice we cannot give a complete proof of Theorem~\ref{thm:effBir} before
proving Theorem~\ref{thm:boundOfCompl}.  Contrary to the exposition of this
survey paper, where ``boundedness of complements'' and ``effective
birationality'' are treated as if they were separate, the proofs of the two
theorems are in fact much intertwined, and this is one of the main points where
they come together.  Many of the results discussed in this overview (``Bound on
anti-canonical volumes'', ``Bound on lc thresholds'') have separate proofs in
the exceptional case.

%
%
\svnid{$Id: 05-effective-birationality.tex 82 2019-01-08 08:34:53Z kebekus $}

\section{Effective birationality}
\label{sec:EB}
\subversionInfo

\subsection{Statement of result}

The second main ingredient in Birkar's proof of boundedness is the following
result.  A proof is sketched in Section~\vref{ssec:BC}.

\begin{thm}[\protect{Effective birationality, \cite[Thm.~1.2]{Bir16a}}]\label{thm:effBir}
  Given $d ∈ ℕ$ and $ε ∈ ℝ^+$, there exists $m ∈ ℕ$ with the following property.
  If $X$ is any $ε$-lc weak Fano variety of dimension $d$, then
  $\lvert -m·K_X \rvert$ defines a birational map.
\end{thm}

\begin{rem}
  The divisors $m·K_X$ in Theorem~\ref{thm:effBir} need not be Cartier.  The
  linear system $\lvert -m·K_X \rvert$ is the space of effective Weil divisors
  on $X$ that are linearly equivalent to $-m·K_X$.
\end{rem}

\subsection{Idea of application}
\label{ssec:EBA}

In the framework of \cite{Bir16a}, effective birationality is used to improve
the boundedness criterion spelled out in Theorem~\ref{thm:boundednessCriterion}
above.

\begin{prop}[\protect{Boundedness criterion, \cite[Prop.~7.13]{Bir16a}}]\label{prop:4.3a}
  Let $d, v ∈ ℕ$ and let $(t_ℓ)_{ℓ ∈ ℕ}$ be a sequence of positive real numbers.
  Let $\mathcal{X}$ be the family of projective varieties $X$ with the following
  properties.
  \begin{enumerate}
  \item The variety $X$ is a klt weak Fano variety of dimension $d$.
    
  \item\label{il:x3} The volume of the canonical class is bounded,
    $\vol(-K_X) ≤ v$.
    
  \item\label{il:x4} For every $ℓ ∈ ℕ$ and every $L ∈ \lvert -ℓ·K_X \rvert$, the
    pair $(X, t_ℓ·L)$ is klt.
  \end{enumerate}
  Then, $\mathcal{X}$ is a bounded family.
\end{prop}

\begin{rem}
  The formulation of Proposition~\ref{prop:4.3a} is meant to illustrate the
  application of Theorem~\ref{thm:effBir} to the boundedness problem.  It is a
  simplified version of Birkar's formulation and defies the logic of his work.
  While we present Proposition~\ref{prop:4.3a} as a corollary to
  Theorem~\ref{thm:effBir}, and to all the results mentioned in
  Section~\ref{sec:bcomp}, Birkar uses \cite[Prop.~7.13]{Bir16a} as one step in
  the inductive proof of ``boundedness of complements'' and ``effective
  birationality''.  That requires him to explicitly list partial cases of
  ``boundedness of complements'' and ``effective birationality'' as assumptions
  to the proposition, and makes the formulation more involved.
\end{rem}

\begin{rem}
  Proposition~\ref{prop:4.3a} reduces the boundedness problem to solving the
  following two problems.
  \begin{itemize}
  \item Boundedness of volumes, as required in \ref{il:x3}.  This is covered in
    the subsequent Section~\ref{sec:volumes}.
  
  \item Existence of numbers $t_ℓ$, as required in \ref{il:x4}.  This amounts to
    bounding ``lc thresholds'' and is covered in Section~\ref{sec:lcthres}.
  \end{itemize}
\end{rem}

To prove Proposition~\ref{prop:4.3a}, Birkar uses effective birationality in the
following form, as a log birational boundedness result.

\begin{prop}[\protect{Log birational boundedness of certain pairs, \cite[Prop.~4.4]{Bir16a}}]\label{thm:lbb}
  Given $d,v ∈ ℕ$ and $ε ∈ ℝ^+$.  Then, there exists $c ∈ ℝ^+$ and a bounded
  family $\cP$ of couples with the following property.  If $X$ is a normal
  projective variety of dimension $d$ and if $B ∈ ℝ\Div(X)$ and $M ∈ ℚ\Div(X)$
  are divisors such that the following holds,
  \begin{enumerate}
  \item the divisor $B$ is effective, with coefficients in $\{0\} ∪ [ε,∞)$,

  \item the divisor $M$ is effective, nef and $|M|$ defines a birational
    map,

  \item the difference $M-(K_X+B)$ is pseudo-effective,

  \item the volume of $M$ is bounded, $\vol(M) < v$,

  \item if $D$ is any component of $M$, then $\mult_D (B+M) ≥ 1$,
  \end{enumerate}
  then there exists a log smooth couple $(X', Σ) ∈ \cP$, a rational map
  $\overline{X} \dasharrow X$ and a resolution of singularities
  $r : \wtilde{X} → X$, with the following
  properties.
  \begin{enumerate}
  \item\label{il:y1} The divisor $Σ$ contains the birational transform on $M$,
    as well as the exceptional divisor of the birational map $β$.

  \item\label{il:y2} The movable part $A_{\wtilde{X}}$ of $r^* M$ is basepoint
    free.
    
  \item\label{il:y3} If $\wtilde{X}'$ is any resolution of $X$ that factors via
    $X'$ and $\wtilde{X}$,
    $$
    \xymatrix{
      \wtilde{X}' \ar[d]_{s\text{, resolution}} \ar[rr]^{\wtilde{β}\text{, birational}} && \wtilde{X} \ar[d]^{r\text{, resolution}} \\
      X' \ar@{-->}[rr]_{β\text{, birational}} && X
    }
    $$
    then the coefficients of the $ℚ$-divisor $s_* (r ◦ \wtilde{β})^* M$ are at
    most $c$ and $\wtilde{β}^* A_{\wtilde{X}}$ is linearly equivalent to zero
    relative to $X'$.
  \end{enumerate}
\end{prop}
\begin{proof}[\protect{Sketch of proof for Proposition~\ref{thm:lbb}, following \cite[p.~42]{Bir16a}}]
  Since $|M|$ defines a birational map, there exists a resolution
  $r : \wtilde{X} → X$ such that $r^* M$ decomposes as the sum of a base point
  free movable part $A_{\wtilde{X}}$ and fixed part $R_{\wtilde{X}}$.  The
  contraction $X → X''$ defined by $A_{\wtilde{X}}$ is birational.  Since
  $\vol(M)$ is bounded, the varieties $X''$ obtained in this way are all members
  of one bounded family $\cP'$.  The family $\cP'$ is however not yet the
  desired family $\cP$, and the varieties in $\cP'$ are not yet equipped with an
  appropriate boundary.  To this end, one needs to invoke a criterion of
  Hacon-McKernan-Xu for ``log birationally boundedness'',
  \cite[Lem.~2.4.2(4)]{MR3034294}, and take an appropriate resolution of the
  elements in $\cP'$.
\end{proof}

\begin{proof}[\protect{Sketch of proof for Proposition~\ref{prop:4.3a}, following \cite[p.~80]{Bir16a}}]
  Applying Theorems~\ref{thm:boundOfCompl} (``Boundedness of complements'') and
  \ref{thm:effBir} (``Effective birationality''), we find a number $m ∈ ℕ$ such
  that every $X ∈ \mathcal{X}$ admits an $m$-complement for $K_X$ and that
  $\lvert -m·K_X \rvert$ defines a birational map.  If $m$-complements $B^+$ of
  $K_X$ could always be chosen such that $(X,B^+)$ were klt, we have seen in
  Observation~\ref{obs:1} that $\mathcal{X}$ is bounded.  However,
  Theorems~\ref{thm:boundOfCompl} guarantees only the existence of an
  $m$-complement $B^+$ of $K_X$ where $(X,B^+)$ is lc.  Using the bounded family
  $\cP$ obtained when applying Proposition~\ref{thm:lbb} with $M = -m·K_X$ and
  $B = 0$, we aim to find a universal constant $ℓ$ and a finite set
  $\mathcal{R}$, and then perturb any given $(X,B^+)$ in order to find a
  boundary $B^{++}$ with coefficients in $\mathcal{R}$ that is $ℚ$-linearly
  equivalent to $-K_X$ and makes $(X, B^{++})$ klt.  Boundedness will then again
  follow from Theorem~\ref{thm:boundednessCriterion}.

  To spell out a few more details of the proof use boundedness of the family
  $\cP$ to infer the existence of a universal constant $ℓ$ with the following
  property.

  \begin{quote}
    If $(X', Σ) ∈ \cP$ and if $A_{X'} ∈ \Div(X')$ is contained in $Σ$ with
    coefficients bounded by $c$, and if $|A_{X'}|$ is basepoint free and defines
    a birational morphism, then there exists $G_{X'} ∈ |ℓ·A_{X'}|$ whose support
    contains $Σ$.
  \end{quote}

  Now assume that one $X ∈ \mathcal{X}$ is given.  It suffices to consider the
  case where $X$ is $ℚ$-factorial and admits an $m$-complement of the form
  $B^+ = \frac{1}{m}·M$, for general $M ∈ \lvert -m·K_X \rvert$.  To make use of
  $ℓ$, consider a diagram as discussed in Item~\ref{il:y3} of
  Proposition~\ref{thm:lbb} above and decompose
  $r^*M = A_{\wtilde{X}} + R_{\wtilde{X}}$ into its moving and its fixed part.
  Write $A := r_* A_{\wtilde{X}}$ and $R := r_* R_{\wtilde{X}}$.
  Item~\ref{il:y1} of Proposition~\ref{thm:lbb} implies that the divisor
  $A_{X'} := s_* \wtilde{β}^* A_{\wtilde{X}}$ is then contained in $Σ$, and
  Item~\ref{il:y3} asserts that it is basepoint free, defines a birational
  morphism.  So, we find $G_{X'} ∈ |ℓ·A_{X'}|$ as above.  Writing
  $G := r_* \wtilde{β}_*s^* G_{X'}$, we find that $G+ℓ·R ∈ |- mℓ·K_X|$, so that
  $(X, t_{mℓ} G)$ is klt by assumption.  We may assume that $t_{mℓ}$ is rational
  and $t_{mℓ} < \frac{1}{mℓ}$.  If $(X, \frac{1}{mℓ}(G+ℓ·R))$ is lc, then set
  $B' := \frac{1}{mℓ}(G+ℓ·R)$.  Otherwise, one needs to use the
  lower-dimensional versions of the variants and generalisations of boundedness
  of complements that we discussed in Section~\ref{sec:vandg} above.  To be more
  precise, using
  \begin{enumerate}
  \item boundedness of complements for generalised polarised pairs for varieties
    of dimension $≤ d-1$, and
    
  \item boundedness of complements in the relative setting for varieties of
    dimension $d$,
  \end{enumerate}
  one can always find a universal number $n$ and $B' ≥ t_{mℓ}·(G+ℓ·R)$ where
  $(X,B')$ is lc and $n·(K_X+B') \sim 0$.  Finally, set
  $$
  B^{++} := \frac{1}{2}·B^+ + \frac{t}{2m}·A - \frac{t}{2mℓ}·G + \frac{1}{2}·B'
  $$
  and then show by direct computation that all required properties hold.
\end{proof}

\subsection{Preparation for the proof of Theorem~\ref*{thm:effBir}}

We prepare for the proof with the following proposition.  In essence, it asserts
that effective divisors with ``degree'' bounded from above cannot have too small
lc thresholds, under appropriate assumptions.  Since this proposition may look
plausible, we do not go into details of the proof.  Further below,
Proposition~\ref{prop:lctbd} gives a substantially stronger result whose proof
is sketched in some detail.

\begin{prop}[\protect{Singularities in bounded families, \cite[Prop.~4.2]{Bir16a}}]\label{p-non-term-places}
  Given $ε' ∈ ℝ^+$ and given a bounded family $\cP$ of couples, there exists a
  number $δ ∈ ℝ^{>0}$ such that the following holds.  Given the the following
  data,
  \begin{enumerate}
  \item an $ε'$-lc, projective pair $(\what{G}, \what{B})$,
    
  \item a reduced divisor $T ∈ \Div(\what{G})$ such that
    $\bigl( \what{G}, \supp (\what{B}+T) \bigr) ∈ \cP$, and
    
  \item an $ℝ$-divisor $\what{N}$ whose support is contained in $T$, and whose
    coefficients have absolute values $≤ δ$,
  \end{enumerate}
  then $(\what{G}, \what{B}+\what{L})$ is klt, for all
  $\what{L} ∈ |\what{N}|_ℝ$.  \qed
\end{prop}

\subsection{Sketch of proof of Theorem~\ref*{thm:effBir}}

Assume that numbers $d$ and $ε$ are given.  Given an $ε$-lc Fano variety $X$ of
dimension $d$, we will be interested in the following two main invariants,
\begin{align*}
  m_X & := \min \{ \: m' ∈ ℕ \,|\, \text{the linear system } \lvert -m'·K_X \rvert \text{ defines a birational map }\} \\
  n_X & := \min \{ \: n' ∈ ℕ\; \,|\, \vol(-n'·K_X) ≥ (2d)^d \:\}
\end{align*}
Eventually, it will turn out that both numbers are bounded from above.  Our aim
here is to bound the numbers $m_X$ by a constant that depends only on $d$ and
$ε$.

\subsection*{Bounding the quotient}

Following \cite{Bir16a}, we will first find an upper bound for the quotients
$m_X/n_X$ by a number that depends only on $d$ and $ε$.

\subsubsection{Construction of non-klt centres}

In the situation at hand, a standard method (``tie breaking'') allows us to find
dominating families of non-klt centres; we refer to
\cite[Sect.~6]{KollarSingsOfPairs} for an elementary discussion, but see also
\cite[Sect.~2.31]{Bir16a}.  Given an $ε$-lc Fano variety $X$ of dimension $d$,
and using the assumption that $\vol(-n_X·K_X) ≥ (2d)^d$, the following has been
shown by Hacon, McKernan and Xu.

\begin{claim}[\protect{Dominating family of non-klt centres, \cite[Lem.~7.1]{MR3224718}}]\label{claim:dfnkc}
  Given any $ε$-lc Fano variety $X$, there exists a dominating family $\cG_X$ of
  subvarieties in $X$ with the following property.  If $(x, y) ∈ X ⨯ X$ is any
  general tuple of points, then there exists a divisor $Δ ∈ |-(n_X+1)·K_X|_ℝ$
  such that the following holds.
  \begin{enumerate}
  \item\label{il:o1} The pair $(X,Δ)$ is not klt at $y$.
    
  \item\label{il:o2} The pair $(X,Δ)$ is lc near $x$ with a unique non-klt
    place.  The associated non-klt centre is a subvariety of the family $\cG_X$.
    \qed
  \end{enumerate}
\end{claim}

Given $X$, we may assume that the members of the families $\cG_X$ all have the
same dimension, and that this dimension is minimal among all families of
subvarieties that satisfy \ref{il:o1} and \ref{il:o2}.

\subsubsection{The case of isolated centres}

If $X$ is given such that the members of $\cG_X$ are points, then the elements
are isolated non-klt centres.  Given $G ∈ \cG_X$, standard vanishing theorems
for multiplier ideals will then show surjectivity of the restriction maps
$$
H⁰ \bigl( X,\, 𝒪_X(K_X+Δ) \bigr) → \underbrace{H⁰ \bigl( G,\,
  𝒪_X(K_X+Δ)|_G \bigr)}_{≅ ℂ}.
$$
In particular, we find that $𝒪_X(K_X+Δ) ≅ 𝒪_X(-n_X·K_X)$ has non-trivial
sections.  Further investigation reveals that a bounded multiple of $-n_X·K_X$
will in fact give a birational map.

\subsubsection{Non-isolated centres}

It remains to consider varieties $X$ where the members of $\cG_X$ are
positive-dimensional.  Following \cite[proofs of Prop.~4.6 and 4.8]{Bir16a}, we
trace the arguments for that case in \emph{very} rough strokes, ignoring all of
the (many) subtleties along the way.  The main observation to handle this case
is the following volume bound.

\begin{claim}[\protect{Volume bound, \cite[Step 3 on p.~48]{Bir16a}}]\label{claim:vB}
  There exists a number $v ∈ ℝ^+$ that depends only on $d$ and $ε$, such that
  for all $X$ and all positive-dimensional $G ∈ \cG_X$, we have
  $\vol(-m_X·K_X|_G) < v$.
\end{claim}
\begin{proof}[Idea of proof for Claim~\ref{claim:vB}]
  Going back and looking at the construction of non-klt centres (that is, the
  detailed proof of Claim~\ref{claim:dfnkc}), one finds that the construction
  can be improved to provide families of lower-dimension centres if only the
  volumes are big enough.  But this collides with our assumption that the
  varieties in $\cG_X$ were of minimal dimension.
  \qedhere~(Claim~\ref{claim:vB})
\end{proof}

To make use of Claim~\ref{claim:vB}, look at one $X$ where the members of
$\cG_X$ are positive-dimensional.  Choose a general divisor\footnote{the divisor
  $M$ should really be taken as the movable part, but we ignore this detail.}
$M ∈ \lvert -m_X·K_X \rvert$, and let $(x,y) ∈ X ⨯ X$ be a general tuple of
points with associated centre $G ∈ \cG_X$.  Since $G$ is a non-klt centre that
has a unique place over it, adjunction (and inversion of adjunction) works
rather well.  Together with the bound on volumes, this allows us to define a
natural boundary $\what{B}$ on a suitable birational modification $\what{G}$ of
the normalisation of $G$, such that the following holds.\CounterStep
\begin{enumerate}
\item The pair $(\what{G}, \what{B})$ is $ε'$-lc, for some controllable number
  $ε'$.

\item Writing $E$ for the exceptional divisor of $\what{G} → G$ and
  $T := (\what{B}+E)_{\red}$, the couple
  $\bigl(\what{G}, \supp (\what{B} + T) \bigr)$ belongs to a bounded family
  $\cP$ that in turn depends only on the numbers $d$ and $ε$.
  
\item The pull-back of $M$ to $\what{G}$ has support in $\supp (\what{B} + T)$.
\end{enumerate}

\subsubsection{End of proof}

The idea now is of course to apply Proposition~\ref{p-non-term-places}, using
the family $\cP$.  Arguing by contradiction, we assume that the numbers
$m_X/n_X$ are unbounded.  We can then find one $X$ where $n_X/m_X$ is really
quite small when compared to the number $δ$ given by
Proposition~\ref{p-non-term-places}.  In fact, taking $\what{N}$ as the
pull-back of $\frac{n_X}{m_X}·M$, it is possible to guarantee that the
coefficients of $\what{N}$ are smaller than $δ$.

Intertwining this proof with the proof of ``boundedness of complements'', we may
use a partial result from that proof, and find $L ∈ |-n_X·K_X|_{ℚ}$, whose
coefficients are $≥ 1$.  Since the points $(x,y) ∈ X ⨯ X$ were chosen
generically, the pull-back $\what{L}$ of $L$ to $\what{G}$ has coefficients
$≥ 1$, and can therefore never appear in the boundary of a klt pair.  But then,
$\what{L} ∈ |\what{N}|_ℝ$, which contradicts Proposition~\ref{p-non-term-places}
and ends the proof.  In summary, we were able to bound the quotient $m_X/n_X$ by
a constant that depends only on $d$ and $ε$.  \qed\mbox{\quad(Boundedness of
quotients)}

\subsection*{Bounding the numbers $m_X$}

Finally, we still need to bound $m_X$.  This can be done by arguing that the
volumes $\vol(-m_X·K_X)$ are bounded from above, and then use the same set of
ideas discussed above, using $X$ instead of a birational model $\what{G}$ of its
subvariety $G$.  Since some of the core ideas that go into boundedness of volumes
are discussed in more detail in the following Section~\ref{sec:volumes} below,
we do not go into any details here.  \qed

%
%
\svnid{$Id: 06-volumina.tex 79 2019-01-07 13:21:51Z kebekus $}

\section{Bounds for volumes}
\label{sec:volumes}
\subversionInfo

\subsection{Statement of result}

Once Theorem~\ref{thm:BAB} (``Boundedness of Fanos'') is shown, the volumes of
anticanonical divisors of $ε$-lc Fano varieties of any given dimension will
clearly be bounded.  Here, we discuss a weaker result, proving boundedness of
volumes for Fanos of dimension $d$, assuming boundedness of Fanos in dimension
$d-1$.

\begin{thm}[\protect{Bound on volumes, \cite[Thm.~1.6]{Bir16a}}]\label{thm:bvol}
  Given $d ∈ ℕ$ and $ε ∈ ℝ^+$, if the $ε$-lc Fano varieties of dimension $d-1$
  form a bounded family, then there is a number $v$ such that $\vol(-K_X) ≤ v$,
  for all $ε$-lc weak Fano varieties $X$ of dimension $d$
\end{thm}

\subsection{Idea of application}

We have seen in Section~\ref{ssec:EBA} how to obtain boundedness criteria for
families of varieties from boundedness of volumes.  This makes
Theorem~\ref{thm:bvol} a key step in the inductive proof of
Theorem~\ref{thm:BAB}.

\subsection{Idea of proof for boundedness of volumes, following \cite[Sect.~9]{Bir16a}}

To illustrate the core idea of proof, we consider only the simplest cases and
make numerous simplifying assumptions, no matter how unrealistic.  The
assumption that $ε$-lc Fano varieties of dimension $d-1$ form a bounded family
will be used in the following form.

\begin{lem}[\protect{Consequence of boundedness, \cite[Lem.~2.22]{Bir16a}}]\label{lem:5-2}
  There exists a finite set $I ⊂ ℝ$ with the following property.  If $X$
  is an $ε$-lc Fano variety of dimension $d-1$, if $r ∈ ℝ^{≥ 1}$ and if
  $D$ is any non-zero integral divisor on $X$ such that $K_X + r·D \equiv 0$,
  then $r ∈ I$.  \qed
\end{lem}

We argue by contradiction and assume that there exists a sequence
$(X_i)_{i ∈ ℕ}$ of $ε$-lc weak Fanos of dimension $d$ such that the sequence
of volumes is strictly increasing, with $\lim \vol(X_i) = ∞$.  For
simplicity of the argument, assume that all $X_i$ are Fanos rather than weak
Fanos, and that they are $ℚ$-factorial.  For the general case, one needs to
consider the maps defined by multiples of $-K_X$ and take small
$ℚ$-factorialisations.

Choose a rational $ε'$ in the interval $(0,ε)$.  Using explicit discrepancy
computations of boundaries of the form $\frac{1}{N}·B'_i$, for
$B'_i ∈ \lvert -N·K_{X_i} \rvert$ general, \cite[Cor.~2.32]{KM98}, we find a
decreasing sequence $(a_i)_{i ∈ ℕ}$ of rationals, with $\lim a_i = 0$, and
boundaries $B_i ∈ ℚ\Div(X_i)$ with the following properties.
\begin{enumerate}
\item For each $i$, the divisor $B_i$ is $ℚ$-linearly equivalent to
  $-a_i·K_{X_i}$.
  
\item The volumes of the $B_i$ are bounded from below, $(2d)^d < \vol(B_i)$.
  
\item\label{il:5-1-3} The pair $(X_i, B_i)$ has total log discrepancy equal to
  $ε'$.
\end{enumerate}
Passing to a subsequence, we may assume that $a_i < 1$ for every $i$.  Again,
discrepancy computation show that this allows us to find sufficiently general,
ample $H_i ∈ ℚ\Div(X_i)$ that are $ℚ$-linearly equivalent to $-(1-a_i)·K_{X_i}$
and have the property that $(X, B_i+H_i)$ are still $ε'$-lc.

Given any index $i$, Item~\ref{il:5-1-3} implies that there exists a prime
divisor $D'_i$ on a birational model $X'_i$ that realises the total log
discrepancy.  For simplicity, consider only the case where one can choose
$X_i = X'_i$ for every $i$, and therefore find prime divisors $D_i$ on $X_i$
that appear in $B_i$ with multiplicity $1-ε'$.  Without that simplifying
assumption one needs to invoke \cite[Cor.~1.4.3]{BCHM10}, in order to replace
the variety $X_i$ by a model that ``extracts'' the divisor $D'_i$.  In summary,
we can write
\begin{equation}\label{eq:5-2-4}
  -K_{X_i} \sim_{ℚ} \frac{1}{a_i}·B_i = \frac{1-ε'}{a_i}·D_i + (\text{effective}).
\end{equation}

As a next step, recall from Remark~\ref{rem:MoriDream} that the $X_i$ are Mori
dream spaces.  Given any $i$, we can therefore run the $-D_i$-MMP, which
terminates with a Mori fibre space on which the push-forward of $D_i$ is
relatively ample.  Again, we ignore all technical difficulties and assume that
$X_i$ itself is the Mori fibre space, and therefore admits a fibration
$X_i → Z_i$ with relative Picard number $ρ(X_i/Z_i) = 1$ such that $D_i$ is
relatively ample.  Let $F_i ⊆ X_i$ be a general fibre.  Adjunction and standard
inequalities for discrepancies imply that $F_i$ is again $ε$-lc and Fano.  The
statement about the relative Picard number implies that any effective divisor on
$X_i$ is either trivial or ample on $F_i$.  In particular,
Equation~\ref{eq:5-2-4} implies that $-K_{F_i} \equiv s_i·D_i$, where
$s_i ≥ \frac{1-ε'}{a_i}$ goes to infinity.  If $\dim F_i = d-1$, or more
generally if $\dim F_i < d$ for infinitely many indices $i$, this contradicts
Lemma~\ref{lem:5-2} and therefore proves Theorem~\ref{thm:bvol}.

It remains to consider the case where the $Z_i$ are points.  Birkar's proof in
this case is similar in spirit to the argumentation above, but technically much
\emph{more} demanding.  He creates a covering family of non-klt centres, uses
adjunction on these centres and the assumption that $ε$-lc Fano varieties of
dimension $d-1$ form a bounded family to obtain a contradiction.  \qed

%
%
\svnid{$Id: 07-lct.tex 79 2019-01-07 13:21:51Z kebekus $}

\section{Bounds for lc thresholds}
\label{sec:lcthres}
\subversionInfo

The last of Birkar's core results presented here pertains to log canonical
thresholds of anti-canonical systems; this is the main result of Birkar's second
paper \cite{Bir16b}.  It gives a positive answer to a well-known conjecture of
Ambro \cite[p.~4419]{MR3556423}.  With the notation introduced in
Section~\ref{sec:singofpairs}, the result is formulated as follows.

\begin{thm}[\protect{Lower bound for lc thresholds, \cite[Thm.~1.4]{Bir16b}}]\label{thm:lct}
  Given $d ∈ ℕ$ and $ε ∈ ℝ^+$, there exists $t ∈ ℝ^+$ with the following
  property.  If $(X,B)$ is any projective $ε$-lc pair of dimension $d$ and if
  $Δ := -(K_X+B)$ is nef and big, then $\lct \bigl(X,\, B,\, |Δ|_ℝ \bigr) ≥ t$.
\end{thm}

Though this is not exactly obvious, Theorem~\ref{thm:lct} can be derived from
boundedness of $ε$-lc Fanos, Theorem~\ref{thm:BAB}.  One of the core ideas in
Birkar's paper \cite{Bir16b} is to go the other way and prove
Theorem~\ref{thm:lct} using boundedness, but only for \emph{toric} Fano
varieties, where the result has been established by Borisov-Borisov in
\cite{MR1166957}.

\subsection{Idea of application}

As pointed out in Section~\ref{ssec:EBA}, bounding lc thresholds from below
immediately applies to the boundedness problem.  To illustration the
application, consider the following corollary, which proves
Theorem~\ref{thm:BAB} in part.

\begin{cor}[Boundedness of $ε$-lc Fanos]\label{cor:bab}
  Given $d ∈ ℕ$ and $ε ∈ ℝ^+$, the family $\mathcal{X}^{\Fano}_{d,ε}$ of
  $ε$-lc Fanos of dimension $d$ is bounded.
\end{cor}
\begin{proof}
  We aim to apply Proposition~\ref{prop:4.3a} to the family
  $\mathcal{X}^{\Fano}_{d,ε}$.  With Theorem~\ref{thm:bvol} (``Bound on
  volumes'') in place, it remains to satisfy Condition~\ref{il:x4} of
  Proposition~\ref{prop:4.3a}: we need a sequence $(t_{ℓ})_{ℓ ∈ ℕ}$ such that
  the following holds.
  \begin{quotation}
    For every $ℓ ∈ ℕ$, for every $X ∈ \mathcal{X}^{\Fano}_{d,ε}$ and every
    $L ∈ \lvert -ℓ·K_X \rvert$, the pair $(X, t_ℓ·L)$ is klt.
  \end{quotation}
  But this is not so hard anymore.  Let $t ∈ ℝ^+$ be the number obtained by
  applying Theorem~\ref{thm:lct}.  Given a number $ℓ ∈ ℕ$, a variety
  $X ∈ \mathcal{X}^{\Fano}_{d,ε}$ and a divisor $L ∈ |-ℓ·K_X|$, observe that
  $\frac{1}{ℓ}·L ∈ |-K_X|_{ℝ}$ and recall from Remark~\vref{rem:sflct} that
  $(X, \frac{t}{2ℓ}·L)$ is klt.  We can thus set $t_{ℓ} := \frac{t}{2ℓ}$.
\end{proof}

\subsection{Preparation for the proof of Theorem~\ref*{thm:lct}: $ℝ$-linear systems of bounded degrees}
\label{ssec:lctbd}
  
To prepare for the proof of Theorem~\ref{thm:lct}, we begin with a seemingly
weaker result that provides bounds for lc thresholds, but only for $ℝ$-linear
systems of bounded degrees.  This result will be used in Section~\ref{ssec:soft}
to prove Theorem~\ref{thm:lct} in an inductive manner.

\begin{prop}[\protect{LC thresholds for $ℝ$-linear systems of bounded degrees, \cite[Thm.~1.6]{Bir16b}}]\label{prop:lctbd}
  Given $d$, $r ∈ ℕ$ and $ε ∈ ℝ^+$, there exists $t ∈ ℝ^+$ with the following
  property.  If $(X,B)$ is any projective, $ε$-lc pair of dimension $d$, if
  $A ∈ \Div(X)$ is very ample with $A-B$ ample and $[A]^d ≤ r$, then
  $\lct \bigl(X,\, B,\, |A|_ℝ \bigr) ≥ t$.
\end{prop}

\begin{rem}\label{rem:6-4}
  The condition on the intersection number, $[A]^d ≤ r$ implies that $X$ belongs
  to a bounded family of varieties.  More generally, if we choose $A$ general in
  its linear system, then $(X,A)$ belongs to a bounded family of pairs.
\end{rem}

The proof of Proposition~\ref{prop:lctbd} is sketched below.  It relies on two
core ingredients.  Because of their independent interest, we formulate them
separately.

\begin{setting}\label{set:6-5}
  Given $d$, $r ∈ ℕ$ and $ε ∈ ℝ^+$, we consider projective, $ε$-lc pairs $(X,B)$
  of dimension $d$ where $X$ is $ℚ$-factorial, equipped with the following
  additional data.
  \begin{enumerate}
  \item A very ample divisor $A ∈ \Div(X)$, with $A-B$ ample and $[A]^d ≤ r$.
    
  \item An effective divisor $L ∈ ℝ\Div(X)$, with $A-L$ ample.

  \item A birational morphism $ν : Y → X$ of normal projective varieties, and
    a prime divisor $T ∈ \Div(Y)$ whose image is a point $x ∈ X$.
\end{enumerate}
\end{setting}

\begin{lem}[\protect{Existence of complements, \cite[Prop.~5.9]{Bir16b}}]\label{lem:A6-6}
  Given $d$, $r ∈ ℕ$ and $ε ∈ ℝ^+$, assume that Proposition~\ref{prop:lctbd}
  holds for varieties of dimension $d-1$.  Then, there exist integers $n$,
  $m ∈ ℕ$ and a real number $0 < ε' < ε$, with the following property.  Whenever
  we are in Setting~\ref{set:6-5}, and whenever there exists a number $t < r$
  such that
  \begin{enumerate}
  \item the pair $(X, B+t·L)$ is $ε'$-lc, and
    
  \item the log discrepancy is realised by $T$, that is
    $a_{\log}(T,X,B+t·L) = ε'$,
  \end{enumerate}
  Then there exists an effective divisor $Λ ∈ ℚ\Div(X)$ such that
  \begin{enumerate}
  \item the divisor $n·Λ$ is integral,
    
  \item the tuple $(X,Λ)$ is lc near $x$, and $T$ is an lc place of $(X,Λ)$, and
    
  \item the divisor $m·A-Λ$ is ample.  \qed
  \end{enumerate}
\end{lem}

\begin{com*}
  Lemma~\ref{lem:A6-6} is another existence-and-boundedness result for
  complements, very much in the spirit of what we have seen in
  Section~\ref{sec:bcomp}.  The relation to complements is made precise in
  \cite[Thm.~1.7]{Bir16b}, which is a core ingredient in Birkar's proof.  In
  fact, after some birational modification of $Y$, Birkar finds a divisor
  $Λ_Y ∈ \Div(Y)$ such that $(Y, Λ_Y)$ is lc near $T$ and such that
  $n·(K_Y + Λ_Y)$ is linearly equivalent to $0$, relative to $X$ and for some
  bounded number $n ∈ ℕ$.  As Birkar points out in \cite[p.~16]{Bir18}, one can
  think of $K_Y + Λ_Y$ as a local-global type of complement.  He then takes $Λ$
  to be the push-forward of $Λ_Y$ and proves all required properties.
\end{com*}

\begin{lem}[\protect{Bound on multiplicity at an lc place, \cite[Prop.~5.7]{Bir16b}}]\label{lem:A6-7}
  Given $d$, $r$ and $n ∈ ℕ$ and $ε ∈ ℝ^+$, assume that
  Proposition~\ref{prop:lctbd} holds for varieties of dimension $\le d-1$.
  Then, there exists $q ∈ ℝ^+$, with the following property.  Whenever we are in
  Setting~\ref{set:6-5}, whenever $a(T,X,B)\le 1$, and whenever a divisor
  $Λ ∈ ℚ\Div(X)$ is given that satisfies the following conditions,
  \begin{enumerate}
  \item $Λ$ is effective and $n·Λ$ is integral,
    
  \item $A-Λ$ is ample,
    
  \item $(X, Λ)$ is lc near $x$, and $T$ is an lc place of $(X,Λ)$,
  \end{enumerate}
  then $T$ appears in the divisor $ν^* L$ with multiplicity
  $\mult_T ν^*L \le q$.  \qed
\end{lem}

\begin{com*}
  Lemma~\ref{lem:A6-7} is perhaps the core of Birkar's paper \cite{Bir16b}.  To
  begin, one needs to realise that the couples $\bigl( X, \supp(Λ) \bigr)$ that
  appear in Lemma~\ref{lem:A6-7} come from a bounded family.  This allows us to
  consider common resolution, and eventually to assume from the outset that
  $(X, Λ)$ is a log-smooth couple.  In particular, $(X, Λ)$ is toroidal, and $T$
  can be obtained by a sequence of blowing ups that are toroidal with respect to
  $(X , Λ)$.  Given that toroidal blow-ups are rather well understood, Birkar
  finds that to bound the multiplicity $\mult_T ν^*L$, it suffices to bound the
  number of blowups involved.

  Bounding the number of blowups is hard, and the next few sentences simplify a
  very complicated argument to the extreme\footnote{see \cite[p.~16f]{Bir18} and
    \cite[Sect.~10]{YXu18} for a more realistic account of all that is
    involved.}.  Birkar establishes a Noether-normalisation theorem, showing
  that he may replace the couple $(X, Λ)$, which is log-smooth, by a pair of the
  form $(ℙ^d, \text{union of hyperplanes})$, which is toric rather than
  toroidal.  Better still, applying surgery coming from the Minimal Model
  Programme, he is then able to replace $Y$ by a toric, Fano, $ε$-lc variety.
  But the family of such $Y$ is bounded by the classic result of
  Borisov-Borisov, \cite{MR1166957}, and a bound for the number of blowups
  follows.
\end{com*}

\begin{proof}[Sketch of proof for Proposition~\ref{prop:lctbd}]
  The proof of Proposition~\ref{prop:lctbd} proceeds by induction, so assume
  that $d$, $r$, and $ε$ are given and that everything was already shown in
  lower dimensions.  Now, given a $d$-dimensional pair $(X, B)$ and a very ample
  $A ∈ \Div(X)$ as in Proposition~\ref{prop:lctbd}, we aim to apply
  Lemma~\ref{lem:A6-6} and \ref{lem:A6-7}.  This is, however, not immediately
  possible because $X$ need not be $ℚ$-factorial.  We know from minimal model
  theory that there exists a small $ℚ$-factorialisation, say $X' → X$, but then
  we need to compare lc thresholds of $X'$ and $X$, and show that the difference
  is bounded.  To this end, recall from Remark~\ref{rem:6-4} that the family of
  all possible $X$ is bounded, which allows us to construct simultaneous
  $ℚ$-factorialisations in stratified families, and hence gives the desired
  bound for the differences.  Bottom line: we may assume that $X$ is
  $ℚ$-factorial.  Let $ε'$ be the number given by Lemma~\ref{lem:A6-6}.

  Next, given any divisor $L ∈ |A|_ℝ$, look at
  $$
  s := \sup \{ s' ∈ ℝ \,|\, (X, B+s'·L) \text{ is $ε'$-lc} \}.
  $$
  Following Remark~\ref{rem:sflct}, we would be done if we could bound $s$ from
  below, independently of $X$, $B$, $A$ and $L$.  To this end, choose a
  resolution of singularities, $ν : Y → X$ and a prime divisor $T ∈ \Div(Y)$
  such that $a_{\log}(T,X,B+s·L) = ε'$.  For simplicity, we will only consider
  the case where $ν(T)$ is a point, say $x ∈ X$ --- if $ν(T)$ is not a point,
  Birkar cuts down with general hyperplanes from $|A|$, uses inversion of
  adjunction and invokes the induction hypothesis in order to proceed.

  In summary, we are now in a situation where we may apply Lemma~\ref{lem:A6-6}
  (``Existence of complements'') to find a divisor $Λ$ and then
  Lemma~\ref{lem:A6-7} (``Bound on multiplicity at an lc place'') to bound the
  multiplicity $\mult_T ν^*L$ from above, independently of $X$, $B$, $A$ and
  $L$.  But then, a look at Definition~\ref{not:logdiscrep} (``log discrepancy'')
  shows that this already gives the desired bound on $s$.
\end{proof}

\subsection{Preparation for the proof of Theorem~\ref*{thm:lct}: varieties of Picard-number one}

The second main ingredient in the proof of Theorem~\ref{thm:lct} is the
following result, which essentially proves Theorem~\ref{thm:lct} in one special
case.  Its proof, which we do not cover in detail, combines all results
discussed in the previous Sections~\ref{sec:bcomp}--\ref{sec:volumes}:
boundedness of complements, effective birationality and bounds for volumes.

\begin{prop}[\protect{Theorem~\ref{thm:lct} in a special case, \cite[Prop~3.1]{Bir16b}}]\label{p-bnd-lct-global-weak}
  Given $d ∈ ℕ$ and $ε ∈ ℝ^+$, assume that Proposition~\ref{prop:lctbd}
  (``LC thresholds for $ℝ$-linear systems of bounded degrees'') holds in
  dimension $≤ d$ and that Theorem~\ref{thm:BAB} (``Boundedness of $ε$-lc
  Fanos'') holds in dimension $≤ d-1$.  Then, there exists $v ∈ ℝ^+$ such
  that the following holds.  If $X$ is any $ℚ$-factorial, $ε$-lc Fano variety of
  dimension $d$ of Picard number one, and if $L ∈ ℝ\Div(X)$ is effective
  with $L \sim_{ℝ} -K_X$, then each coefficient of $L$ is less than or equal to
  $v$.  \qed
\end{prop}

\subsection{Sketch of proof of Theorem~\ref*{thm:lct}}
\label{ssec:soft}

Like other statements, Theorem~\ref{thm:lct} is shown using induction over the
dimension.  The following key lemma provides the induction step.

\begin{lem}[\protect{Implication Proposition~\ref{prop:lctbd} $⇒$ Theorem~\ref{thm:lct}, \cite[Lem.~3.2]{Bir16b}}]\label{l-local-lct-bab-to-global-lct}
  Given $d ∈ ℕ$, assume that Proposition~\ref{prop:lctbd} (``LC thresholds
  for $ℝ$-linear systems of bounded degrees'') holds in dimension $≤ d$ and that
  Theorem~\ref{thm:BAB} (``Boundedness of $ε$-lc Fanos'') holds in dimension
  $≤ d-1$.  Then, Theorem~\ref{thm:lct} (``Lower bound for lc thresholds'')
  holds in dimension $d$.
\end{lem}
\begin{proof}[\protect{Sketch of proof following \cite[p.~13f]{Bir16b}}]
  The first steps in the proof are similar to the proof of
  Proposition~\ref{prop:lctbd}.  Choose any number $ε'∈ (0,ε)$.  Given any
  projective, $d$-dimension, $ε$-lc pair $(X,B)$ be as in Theorem~\ref{thm:lct}
  in dimension $d$ and any divisor $L ∈ |Δ|_{ℝ}$, let $s$ be the largest number
  such that $(X,B+s·L)$ is $ε'$-lc.  We need to show $s$ is bounded from below
  away from zero.  In particular, we may assume that $s<1$.  As in the proof of
  Proposition~\ref{prop:lctbd}, we may also assume $X$ is $ℚ$-factorial.  There
  is a birational modification $φ : Y → X$ and a prime divisor $T ∈ \Div(Y)$
  with log discrepancy
  \begin{equation}\label{eq:ghjfg}
    a_{\log}(T,X,B+s·L)=ε'.
  \end{equation}
  Techniques of \cite{BCHM10} (``extracting a divisor'') allow us to assume that
  $φ$ is either the identity, or that the $φ$-exceptional set equals $T$
  precisely.  The assumption that $X$ is $ℚ$-factorial allows us to pull back
  divisors.  Let
  $$
  B_Y := φ^*(K_X+B)-K_Y \quad\text{and}\quad L_Y := φ^*L.
  $$
  Using the definition of log discrepancy, Definition~\vref{not:logdiscrep}, the
  assumption that $(X,B)$ is $ε$-lc and Equation~\eqref{eq:ghjfg} are formulated
  in terms of divisor multiplicities as
  $$
  \mult_T B_Y ≤ 1-ε \quad\text{and}\quad \mult_T(B_Y+s·L_Y) = 1-ε',
  $$
  hence $\mult_T (s·L_Y) ≥ ε-ε'$.
  
  The pair $(Y, B_Y+ s·L_Y)$ is klt and weak log Fano, which implies that $Y$ is
  Fano type.  Recalling from Remark~\vref{rem:MoriDream} that $Y$ is thus a Mori
  dream space, we may run a $(-T)$-Minimal Model Programme and obtain rational
  maps,
  $$
  \xymatrix{ %
    Y \ar@{-->}[rrrr]^{α\text{, extr.\ contractions and flips}} &&&& Y' \ar[rrrr]^{β\text{, Mori fibre space}} &&&& Z',
  }
  $$
  where $-T$ is ample when restricted to general fibres of $β$.  We write
  $B_{Y'} := α_* B_Y$ and $L_{Y'} := α_* L_Y$ and note that
  $$
  -(K_{Y'}+B_{Y'}+s·L_{Y'}) \overset{\text{def.\ of $L$}}{\sim_ℝ} (1-s)L_{Y'} \overset{s < 1}{≥} 0.
  $$
  Moreover, an explicit discrepancy computation along the lines of
  \cite[Cor.~2.32]{KM98} shows that $(Y', B_{Y'}+s·L_{Y'})$ is $ε'$-lc, because
  $(Y,B_Y+s·L_Y)$ is $ε'$-lc and because $-(K_Y+B_Y+s·L_Y)$ is semiample.  There
  are two cases now.

  If $\dim Z' > 0$, then restricting to a general fibre of $Y' → Z'$ and
  applying Proposition~\ref{prop:lctbd} (``LC thresholds for $ℝ$-linear systems
  of bounded degrees'') in lower dimension\footnote{or applying
    Theorem~\ref{thm:BAB} (``Boundedness of $ε$-lc Fanos'')} shows that the
  coefficients of those components of $(1-s)·L_{Y'}$ that dominate $Z'$
  components of are bounded from above.  In particular, $\mult_{T'}(1-s)·L_{Y'}$
  is bounded from above.  Thus from the inequality
  $$
  \mult_{T'}(1-s)·L_{Y'} ≥ \frac{(1-s)·(ε-ε')}{s},
  $$
  we deduce that $s$ is bounded from below away from zero.

  If $Z'$ is a point, then $Y'$ is a Fano variety with Picard number one.  Now
  $$
  -K_{Y'} \sim_ℝ (1-s)·L_{Y'} + B_{Y'} + s·L_{Y'}≥ (1-s)·L_{Y'},
  $$
  so by Proposition~\ref{p-bnd-lct-global-weak}, $\mult_{T'}(1-s)·L_{Y'}$ is bounded
  from above which again gives a lower bound for $s$ as before.
\end{proof}

%
%
\svnid{$Id: 08-jordan.tex 84 2019-04-08 23:01:11Z kebekus $}

\section{Application to the Jordan property}
\label{sec:jordan}
\subversionInfo

We explain in this section how the boundedness result for Fano varieties applies
to the study of birational automorphism groups, and how it can be used to prove
the Jordan property.  Several of the core ideas presented here go back to work
of Serre, who solved the two dimensional case, \cite[Thm.~5.3]{MR2567402} but
see also \cite[Thm.~3.1]{MR2648675}.  If one is only interested in the
three-dimensional case, where birational geometry is particularly
well-understood, most arguments presented here can be simplified.

\subsection{Existence of subgroups with fixed points}

If $X$ is any rationally connected variety, Theorem~\ref{thm:jordan} (``Jordan
property of Cremona groups'') asks for the existence of finite Abelian groups in
the Cremona groups $\Bir(X)$.  As we will see in the proof, this is almost
equivalent to asking for finite groups of automorphisms that admit fixed points,
and boundedness of Fanos is the key tool used to find such groups.  The
following lemma is the simplest result in this direction.  Here, boundedness
enters in a particularly transparent way.

\begin{lem}[\protect{Fixed points on Fano varieties, \cite[Lem.~4.6]{MR3483470}}]\label{lem:j1}
  Given $d ∈ ℕ$, there exists a number $j^{\Fano}_d ∈ ℕ$ such that for any
  $d$-dimensional Fano variety $X$ with canonical singularities and any finite
  subgroup $G ⊆ \Aut(X)$, there exists a subgroup $F ⊆ G$ of index
  $|G:F| ≤ j^{\Fano}_d$ acting on $X$ with a fixed point.
\end{lem}

\begin{rem}
  To keep notation simple, Lemma~\ref{lem:j1} is formulated for Fanos with
  canonical singularities, which is the relevant case for our application.  In
  fact, it suffices to consider Fanos that are $ε$-lc.
\end{rem}

\begin{proof}[Proof of Lemma~\ref{lem:j1}]
  As before, write $\mathcal{X}^{\Fano}_{d,0}$ for the $d$-dimensional Fano
  variety $X$ with canonical singularities.  It follows from boundedness,
  Theorem~\ref{thm:BAB} or Corollary~\ref{cor:bab}, that there exist numbers
  $m$, $v ∈ ℕ$ such that the following holds for every
  $X ∈ \mathcal{X}^{\Fano}_{d,0}$.
  \begin{enumerate}
  \item The divisor $-m·K_X$ is Cartier and very ample.
  \item The self-intersection number of $-m·K_X$ is bounded by $v$.  More
    precisely,
    $$
    -[m·K_X]^{d} ≤ v.
    $$
  \end{enumerate}
  Given $X$, observe that the associated line bundles $𝒪_X(-m·K_X)$ are
  $\Aut(X)$-linearised.  Accordingly, there exists a number $N ∈ ℕ$, such that
  every $X ∈ \mathcal{X}^{\Fano}_{d,0}$ admits an $\Aut(X)$-equivariant
  embedding $X ↪ ℙ^N$.  Let $j^{\Jordan}_{N+1}$ be the number obtained by
  applying the classical result of Jordan, Theorem~\ref{thm:jordan-lin}, to
  $\GL_{N+1}(ℂ)$, and set $j^{\Fano}_d := j^{\Jordan}_{N+1}·v$.

  Now, given any $X ∈ \mathcal{X}^{\Fano}_{d,0}$ and any finite subgroup
  $G ⊆ \Aut(X)$, the $G$ action extends to $ℙ^N$.  The action is thus induced by
  a representation of a finite linear group $Γ$, say
  $$
  \xymatrix{ %
    Γ \ar@{^(->}[r] \ar@{->>}[d] & \GL_{N+1}(ℂ) \ar@{->>}[d] \\
    G \ar@{^(->}[r] & \PGL_{N+1}(ℂ).  \\
  }
  $$
  By Theorem~\ref{thm:jordan-lin}, the classic result of Jordan, we find a
  finite Abelian subgroup $Φ ⊆ Γ$ of index $|Φ:Γ| ≤ j^{\Jordan}_{N+1}$.  Since
  $Φ$ is Abelian, the $Φ$-representation space $ℂ^{N+1}$ is a direct sum of
  one-dimensional representations.  Equivalently, we find $N+1$ linearly
  independent, $Φ$-invariant, linear hyperplanes $H_i ⊂ ℙ^{N+1}$.  The
  intersection of suitably chosen $H_i$ with $X$ is then a finite, $Φ$-invariant
  subset $\{x_1, …, x_r\} ⊂ X$, of cardinality $r ≤ v$.  The stabiliser of
  $x_1 ∈ X$ is a subgroup $Φ_{x_1} ⊂ Φ$ of index $|Φ:Φ_{x_1}| ≤ v$.  Taking $F$
  as the image of $Φ_{x_1} → G$, we obtain the claim.
\end{proof}

\begin{rem}
  The proof of Lemma~\ref{lem:j1} shows that the groups $G$ are close to
  Abelian.  It also gives an estimate for $j^{\Fano}_d$ in terms of the volume
  bound (``$v$'') and the classical Jordan constant $j^{\Fano}_d$.
\end{rem}

As a next step, we aim to generalise the results of Lemma~\ref{lem:j1} to
varieties that are rationally connected, but not necessarily Fano.  The
following result makes this possible.

\begin{lem}[\protect{Rationally connected subvarieties on different models, \cite[Lem.~3.9]{MR3483470}}]\label{lem:xfgs}
  Let $X$ be a projective variety with an action of a finite group $G$.  Suppose
  that $X$ is klt, with $Gℚ$-factorial singularities and let
  $f : X \dasharrow Y$ be a birational map obtained by running a $G$-Minimal
  Model Programs.  Suppose that there exists a subgroup $F ⊂ G$ and an
  $F$-invariant, rationally connected subvariety $T ⊊ Y$.  Then, there exists an
  $F$-invariant rationally connected subvariety $Z ⊊ X$.  \qed
\end{lem}

Since we are mainly interested to see how boundedness applies to birational
transformation groups, we will not explain the proof of Lemma~\ref{lem:xfgs} in
detail.  Instead, we merely list a few of the core ingredients, which all come
from minimal model theory and birational geometry.
\begin{itemize}
\item Hacon-McKernan's solution \cite{HMcK07} to Shokurov's ``rational
  connectedness conjecture'', which guarantees in essence that the fibres of all
  morphisms appearing in the MMP are rationally chain connected.

\item A fundamental result of Graber-Harris-Starr, \cite{GHS03}, which implies
  that if $f : X → Y$ is any dominant morphism of proper varieties, where both
  the target $Y$ and a general fibre is rationally connected, then $X$ is also
  rationally connected.
  
\item Log-canonical centre techniques, in particular a relative version of
  Kawamata's subadjunction formula, \cite[Lem.~2.5]{MR3483470}.  These results
  identify general fibres of minimal log-canonical centres under contraction
  morphisms as rationally connected varieties of Fano type.
\end{itemize}

\begin{prop}[\protect{Fixed points on rationally connected varieties, \cite[Lem.~4.7]{MR3483470}}]\label{prop:j2}
  Given $d ∈ ℕ$, there exists a number $j^{rc}_d ∈ ℕ$ such that for any
  $d$-dimensional, rationally connected projective variety $X$ and any finite
  subgroup $G ⊆ \Aut(X)$, there exists a subgroup $F ⊆ G$ of index
  $|G:F| ≤ j^{rc}_d$ acting on $X$ with a fixed point.
\end{prop}
\begin{proof}[Sketch of proof]
  We argue by induction on the dimension.  Since the case $d=1$ is trivial,
  assume that $d > 1$ is given, and that numbers $j^{rc}_1, …, j^{rc}_{d-1}$
  have been found.  Set
  $$
  j_d := \max \{j^{rc}_1, …, j^{rc}_{d-1}, j^{\Fano}_d \}
  \quad\text{and}\quad j^{rc}_d := (j_d)².
  $$

  Assume that a $d$-dimensional, rationally connected projective variety $X$ and
  a finite subgroup $G ⊆ \Aut(X)$ are given.  By induction hypothesis, it
  suffices to find a subgroup $G' ⊆ G$ of index $|G:G'| ≤ j_d$ and a
  $G'$-invariant, rationally connected, proper subvariety $X' ⊊ X$.

  If $\wtilde{X} → X$ is the canonical resolution of singularities, as in
  \cite{BM96}, then $\wtilde{X}$ is likewise rationally connected, $G$ acts on
  $\wtilde{X}$ and the resolution morphism is equivariant.  Since images of
  rationally connected, invariant subvarieties are rationally connected and
  invariant, we may assume from the outset that $X$ is smooth.  But then we can
  run a $G$-equivariant Minimal Model Programme\footnote{The existence of an MMP
    terminating with a fibre space is \cite[Cor.~1.3.3]{BCHM10}, which we have
    quoted before.  The fact that the MMP can be chosen in an equivariant manner
    is not explicitly stated there, but follows without much pain.} terminating
  with a $G$-Mori fibre space,
  $$
  \xymatrix{ %
    X \ar@{-->}[rrr]^{G\text{-equivariant MMP}} &&& X' \ar[rrr]^{G\text{-Mori fibre space}} &&& Y.
  }
  $$
  In the situation at hand, Lemma~\ref{lem:xfgs} claims that to find proper,
  invariant, rationally connected varieties on $X$, it is equivalent to find
  them on $X'$.  The fibre structure, however, makes that feasible.

  Indeed, if the base $Y$ of the fibration happens to be a point, then $X'$ is
  Fano with terminal singularities, and Lemma~\ref{lem:j1} applies.  Otherwise,
  let $G_Y$ be the image of $G$ in $\Aut(Y)$, let $G_{X'/Y} ⊆ G$ be the
  ineffectivity of the $G$-action on $Y$, and consider the exact sequence
  $$
  1 → G_{X'/Y} → G → G_Y → 1.
  $$
  As the image of the rationally connected variety $X'$, the base $Y$ is itself
  rationally connected.  By induction hypothesis, using that $\dim Y < \dim X$,
  there exists a subgroup $F'_Y ⊆ G_Y$ of index $|G_Y:F'_Y| < j_d$ that acts on
  $Y$ with a fixed point, say $y ∈ Y$.  Let $G' ⊂ G$ be the preimage of $G'_Y$.
  The fibre $X_y$ is then invariant with respect to the action of $G'$ and
  rationally chain connected by \cite[Cor.~1.3]{HMcK07}.  Better still,
  Prokhorov and Shramov show that it contains a rationally connected,
  $G'$-invariant subvariety.  The induction applies.
\end{proof}

\subsection{Proof of Theorem~\ref*{thm:jordan} (``Jordan property of Cremona groups'')}
\label{ssec:prof}
  
Given a number $d ∈ ℕ$, we claim that the number $j := j^{rc}_d·j^{\Jordan}_d$
will work for us, where $j^{rc}_d$ is the number found in
Proposition~\ref{prop:j2}, and $j^{\Jordan}_d$ comes from Jordan's
Theorem~\ref{thm:jordan-lin}.  To this end, let $X$ be any rationally connected
variety of dimension $d$, and let $G ⊆ \Bir(X)$ be any finite group.  Blowing up
the indeterminacy loci of the birational transformations $g ∈ G$ in an
appropriate manner, we find a birational, $G$-equivariant morphism
$\wtilde{X} → X$ where the action of $G$ in $\wtilde{X}$ is regular rather than
merely birational, see \cite[Thm.~3]{MR0337963}.  Combining with the canonical
resolution of singularities, we may assume that $\wtilde{X}$ is smooth.
Proposition~\ref{prop:j2} will then guarantee the existence of a subgroup
$G' ⊆ G$ of index $|G:G'| ≤ j^{rc}_d$ acting on $\wtilde{X}$ with a fixed point
$\wtilde{x}$.  Standard arguments (``linearisation at a fixed point'') that go
back to Minkowski show that the induced action of $G'$ on the Zariski tangent
space $T_{\wtilde{x}}(\wtilde{X})$ is faithful, so that Jordan's
Theorem~\ref{thm:jordan-lin} applies.  In fact, assuming that there exists an
element $g ∈ G' ∖ \{e\}$ with
$Dg|_{\wtilde{x}} = \Id_{T_{\wtilde{x}}(\wtilde{X})}$, choose coordinates and
use a Taylor series expansion to write
$$
g(\vec{x}) = \vec{x} + A_k(\vec{x}) + A_{k+1}(\vec{x}) + …
$$
where each $A_m(\vec{x})$ is homogeneous of degree $m$, and $A_k$ is non-zero.
Given any number $n$, observe that
$$
g^n(\vec{x}) = \vec{x} + n·A_k(\vec{x}) + \text{(higher order terms)}.
$$
Since the base field has characteristic zero, this contradicts the finite order
of $g$.  \qed

\ifdefined\subtitle
\else
\fi

\end{document}